\documentclass[english]{smfart}

\usepackage{bull}
\usepackage{amsfonts}
\usepackage{amsmath}
\usepackage{amsthm}
\usepackage{todonotes}
\usepackage[utf8]{inputenc}
\usepackage{amssymb}
\usepackage{url}
\usepackage[british]{babel}
\usepackage[all,cmtip]{xy}

\newtheorem{theorem}{Theorem}
\numberwithin{theorem}{section}

\newtheorem{corollary}[theorem]{Corollary}

\newtheorem{definition}[theorem]{Definition}

\newtheorem{lemma}[theorem]{Lemma}

\newtheorem{proposition}[theorem]{Proposition}

\theoremstyle{definition}
\newtheorem{remark}[theorem]{Remark}

\newcommand{\abGal}[1] {\operatorname{Gal}\big(\overline{#1}/#1\big)}
\newcommand{\fieldExtension}{K:E^*}
\newcommand{\MT}{\operatorname{MT}}
\newcommand{\Hg}{\operatorname{Hg}}

\date{}
\begin{document}

\title{Galois representations attached to abelian varieties of CM type}
\alttitle{Représentations galoisiennes associées aux variétés abéliennes de type CM}

\author{Davide Lombardo}
\address{Département de Mathématiques d'Orsay}
 \email{davide.lombardo@math.u-psud.fr}

\begin{abstract}
Let $K$ be a number field, $A/K$ be an absolutely simple abelian variety of CM type, and $\ell$ be a prime number. We give explicit bounds on the degree over $K$ of the division fields $K(A[\ell^n])$, and when $A$ is an elliptic curve we also describe the full Galois group of $K(A_{\text{tors}})/K$. This makes explicit previous results of Serre \cite{MR0387283} and Ribet \cite{MR608640}, and strengthens a theorem of Banaszak, Gajda and Kraso\'n \cite{MR1971250}. Our bounds are especially sharp when the CM type of $A$ is nondegenerate.
\end{abstract}

\begin{altabstract}
Soient $K$ un corps de nombres, $A/K$ une variété abélienne géométriquement simple de type CM et $\ell$ un nombre premier. Nous donnons des bornes explicites sur le degré sur $K$ des extensions $K(A[\ell^n])$ engendrées par les points de $\ell^n$-torsion de $A$, et quand $A$ est une courbe elliptique nous décrivons le groupe de Galois de $K(A_{\text{tors}})/K$ tout entier. Cela fournit une version explicite de résultats antérieurs de Serre \cite{MR0387283} et Ribet \cite{MR608640}, et renforce un théorème de Banaszak, Gajda and Kraso\'n \cite{MR1971250}. Nos bornes sont particulièrement fines quand le type CM de $A$ est non-dégénéré.
\end{altabstract}

\altkeywords{Multiplication complexe, représentations galoisiennes, courbes elliptiques, groupe de Mumford-Tate}

\maketitle

\section{Introduction and statement of the result}
The aim of this work is to study division fields of simple abelian varieties of CM type. Recall that an abelian variety $A$, of dimension $g$ and defined over a number field $K$, is said to admit (potential) complex multiplication, or CM for short, if there is an embedding $E \hookrightarrow \operatorname{End}_{\overline{K}}(A) \otimes \mathbb{Q}$, where $E$ is an étale $\mathbb{Q}$-algebra of degree $2g$. We shall very often restrict to the situation of $A$ admitting complex multiplication by $E$ \textit{over $K$}, by which we mean that $\operatorname{End}_K(A)$ is equal to $\operatorname{End}_{\overline{K}}(A)$, and of $A$ being absolutely simple, or equivalently, of $E$ being a number field (of degree $2g$ over $\mathbb{Q}$). The problem we discuss is that of estimating the degree $[K(A[\ell^n]):K]$, where $\ell$ is a prime number and $K(A[\ell^n])$ is the field generated over $K$ by the coordinates of the $\ell^n$-torsion points of $A$ in $\overline{K}$. 
As we shall see shortly, this is really a problem in the theory of Galois representations, and the seminal contributions of Shimura--Taniyama \cite{MR0125113} and Serre--Tate \cite{MR0236190} provide us with powerful tools for handling these representations in the CM case. Employing such tools, Silverberg studied in \cite{MR971328} the extension of $K$ generated by a single torsion point of $A$, while Ribet gave in \cite{MR608640} asymptotic (non-effective) bounds on $[K(A[\ell^n]):K]$ as $n \to \infty$. Our first result can be seen as an explicit version of the main theorem of \cite{MR608640}:

\begin{theorem}\label{thm_Simple}
Let $K$ be a number field and $A/K$ be an abelian variety of dimension $g$ admitting complex multiplication over $K$ by an order in the CM field $E$. Denote by $\mu$ be the number of roots of unity contained in $E$ and by $h(K)$ the class number of $K$. Let $r$ be the rank of the Mumford-Tate group of $A$ (cf. definition \ref{def_MT}) and $\ell > \sqrt{2 \cdot g!}$ be a prime unramified in $E \cdot K$. The following inequality holds:
\[
\frac{1}{4\mu \sqrt{g!}} \cdot \ell^{nr} \leq [K(A[\ell^n]):K] \leq \frac{5 }{2} \mu \cdot h(K) \cdot \ell^{nr}.
\]
\end{theorem}

Even though theorem \ref{thm_Simple} gives a good idea of the actual order of magnitude of the degree $[K(A[\ell^n]):K]$, we can in fact prove much more precise results that apply to all primes $\ell$ and which are most easily described in the language of Galois representations. Recall that for every $\ell$ and every $n$ there is a natural continuous action of $\abGal{K}$ on $A[\ell^n]$, giving rise to a representation
\[
\rho_{\ell^n} : \abGal{K} \to \operatorname{Aut} (A[\ell^n]);
\]
the extension $[K(A[\ell^n]):K]$ is Galois, and its Galois group can be identified with the image $G_{\ell^n}$ of $\rho_{\ell^n}$. Taking the inverse limit of this system of representations gives rise to the $\ell$-adic representation on the Tate module $T_\ell A$,
\[
\rho_{\ell^\infty} : \abGal{K} \to \operatorname{Aut} (T_\ell A).
\]

We denote by $G_{\ell^\infty}$ the image of $\rho_{\ell^\infty}$ and remark that, for every $n$, the group $G_{\ell^n}$ is clearly isomorphic to the image of $G_{\ell^\infty}$ through the canonical projection
\[
\operatorname{Aut} (T_\ell A) \to \operatorname{Aut} \left ( \frac{T_\ell A}{\ell^n T_\ell A}\right) \cong \operatorname{Aut} (A[\ell^n]);
\]
for simplicity of exposition, we fix once and for all a $\mathbb{Z}_\ell$-basis of $T_\ell A$ and consider $G_{\ell^\infty}$ (resp. $G_{\ell^n}$) as a subgroup of $\operatorname{GL}_{2g}(\mathbb{Z}_\ell)$ (resp. of $\operatorname{GL}_{2g}(\mathbb{Z}/\ell^n \mathbb{Z})$).

We have thus reduced the problem of giving bounds on $[K(A[\ell^n]):K]$ to that of describing $G_{\ell^n}$: in trying to do so, it is natural to compare $G_{\ell^\infty}$ with $\MT(A)$, the Mumford-Tate group of $A$ (cf. definition \ref{def_MT}). By construction, $\MT(A)$ is an algebraic subtorus of $\operatorname{GL}_{2g}$ which is only defined over $\mathbb{Q}$, so there is no obvious good definition for the group of its $\mathbb{Z}_\ell$-valued points. However, Ono \cite{MR0124326} has shown that there is in fact a good notion of $\MT(A)(\mathbb{Z}_\ell)$ (cf. definition \ref{def_ZlPoints}), and the Mumford-Tate conjecture \cite[§4]{MR0206003} -- which is a theorem for CM abelian varieties (\cite{MR0228500} and \cite{MR0125113}) -- can be expressed by saying that, possibly after replacing $K$ by a finite extension, $G_{\ell^\infty}$ is a finite-index subgroup of $\MT(A)(\mathbb{Z}_\ell)$.
For the sake of simplicity, assume for now that no extension of the base field $K$ is necessary to attain the condition $G_{\ell^\infty} \subseteq \MT(A)(\mathbb{Z}_\ell)$ (our results do not depend on this assumption). The problem of estimating the degree $[K(A[\ell^n]):K]$ is then reduced to the study of two separate quantities: the order of the finite group $\MT(A)(\mathbb{Z}/\ell^n\mathbb{Z})$ and the index $[\MT(A)(\mathbb{Z}_\ell):G_{\ell^\infty}]$.

We treat the first problem in two important situations: when $\ell$ is unramified in $E$ (a rather simple case, covered by lemma \ref{lemma_Nicolas}), and when the CM type of $A$ is nondegenerate (theorem \ref{thm_Nondegenerate}). Our result can be stated as follows:
\begin{theorem}\label{thm_Nondegenerate_Intro}
Let $A/K$ be an absolutely simple abelian variety of dimension $g$, admitting (potential) complex multiplication by the CM field $E$. Denote by $\MT(A)$ the Mumford-Tate group of $A$ and let $r$ be its rank.
\begin{enumerate}
\item If $\ell$ is unramified in $E$ the following inequalities hold:
\[
(1-1/\ell)^{r} \ell^{nr}  \leq \left|\MT(A)(\mathbb{Z}/\ell^n\mathbb{Z})\right| \leq (1+1/\ell)^{r} \ell^{nr}.
\]
\item Suppose $r=g+1$. For all primes $\ell \neq 2$ and all $n \geq 1$ we have
\[
(1-1/\ell)^{g+1} \cdot \ell^{(g+1)n} \leq |\MT(A)(\mathbb{Z}/\ell^n\mathbb{Z})| \leq 2^g \left(1 + 1/\ell \right)^{g-1} \ell^{(g+1)n},
\]
while for $\ell=2$ and all $n \geq 1$ we have
\[
\frac{1}{2^{2g+3}} \cdot 2^{(g+1)n} \leq |\MT(A)(\mathbb{Z}/2^n\mathbb{Z})| \leq 2^{2g-1} \cdot 2^{(g+1)n}.
\]
\end{enumerate}
\end{theorem}

As for the index $[\MT(A)(\mathbb{Z}_\ell):G_{\ell^\infty}]$, our main result is as follows (cf. definition \ref{def_ReflexNorm} for the notion of reflex norm):

\begin{theorem}\label{thm_Finale_Intro}{(Theorem \ref{thm_Finale})}
Let $A/K$ be an absolutely simple abelian variety of dimension $g$ admitting complex multiplication over $K$ by the CM type $(E,S)$, and let $\ell$ be a prime number. If $A$ has bad reduction at a place of $K$ dividing $\ell$ let $\mu^*=|\mu(E)|$, the number of roots of unity in $E$; if on the contrary $A$ has good reduction at all places of $K$ of characteristic $\ell$ set $\mu^*=1$.
Denote by $r$ the rank of $\MT(A)$ and by $F$ the group of connected components of the kernel of the reflex norm $T_{E^*} \to T_{E}$, where $E^*$ is the reflex field of $E$. Then:

\begin{enumerate}
\item[(1)] The index $\left[G_{\ell^\infty} : G_{\ell^\infty} \cap \MT(A)(\mathbb{Z}_\ell) \right]$ does not exceed $|\mu(E)| \cdot h(K)$, where $h(K)$ is the class number of $K$.
\item[(2)] We have $\left[\MT(A)(\mathbb{Z}_\ell): G_{\ell^\infty} \cap \MT(A)(\mathbb{Z}_\ell) \right] \leq \mu^* \cdot [\fieldExtension] \cdot |F|^{2r}$.
\item[(3)] If $\ell$ is unramified in $E$ and does not divide $|F|$, then the index $\left[\MT(A)(\mathbb{Z}_\ell): G_{\ell^\infty} \cap \MT(A)(\mathbb{Z}_\ell)\right]$ divides $\mu^* \cdot [\fieldExtension] \cdot |F|$. If $\ell$ is also unramified in $K$, the bound can be improved to $\mu^* \cdot |F|$.
\end{enumerate}
Finally we have $r\leq g+1$ and $|F| \leq f(r) \leq f(g+1)$, where \[\displaystyle 
f(x)=\left\lfloor 2 \left( \frac{x+1}{4} \right)^{(x+1)/2} \right\rfloor.\]
\end{theorem}

\begin{remark} A few comments are in order:
\begin{itemize}
\item Theorem \ref{thm_Simple} follows immediately upon combining theorems \ref{thm_Nondegenerate_Intro} and \ref{thm_Finale_Intro}.
\item The assumption that the action of $E$ is defined over $K$ implies that the reflex field $E^*$ is contained in $K$, see \cite[Chap. 3, Theorem 1.1]{MR713612}. In particular, the degree $[K:E^*]$ makes sense.
\item The condition $\ell \nmid |F|$ is certainly satisfied if $\ell > |F|$: in particular, it is true for all primes $\ell > f(r)$.
\item Since $|F|$ is bounded by $f(g+1)$, the degree $[\fieldExtension]$ does not exceed $[K:\mathbb{Q}]$, and $\mu^*$ can be controlled in terms of $g$ alone (a trivial bound is for example $\mu^* \leq 16g^2$), we see that part (2) of theorem \ref{thm_Finale_Intro} gives a universal bound on $\left[\MT(A)(\mathbb{Z}_\ell): G_{\ell^\infty} \cap \MT(A)(\mathbb{Z}_\ell)\right]$ that only depends on $g$ and $[K:\mathbb{Q}]$.
\item
For small values of $g$ the function $f(g+1)$ takes reasonably small values: we have $f(3)=2$, $f(4)=3$, $f(5)=6$, $f(6)=14$ and $f(7)=32$.
\end{itemize}
\end{remark}

In the special case of elliptic curves the Mumford-Tate group admits a particularly simple description, which leads to a very precise characterization of the corresponding Galois representation. Such a description can already be found (in a non-effective form) in \cite[Corollaire on p.302]{MR0387283}, and the following result makes it completely explicit:

\begin{theorem}{(Theorem \ref{thm_EC})}
Let $A/K$ be an elliptic curve such that $\operatorname{End}_{\overline K}(A)$ is an order in an imaginary quadratic field $E$. Denote by $\displaystyle 
\rho_\infty : \abGal{K} \to \prod_{\ell} \operatorname{Aut} T_\ell A$ 
the natural adelic representation attached to $A$, and let $G_\infty$ be its image. For every prime $\ell$ denote by $C_\ell$ the group $\left(\mathcal{O}_E \otimes \mathbb{Z}_\ell \right)^\times$, considered as a subgroup of $\operatorname{Aut}_{\mathbb{Z}_\ell} \left( \mathcal{O}_E \otimes \mathbb{Z}_\ell \right) \cong \operatorname{GL}_2(\mathbb{Z}_\ell) \cong \operatorname{Aut} T_\ell A$, 
and let $N(C_\ell)$ be the normalizer of $C_\ell$ in $\operatorname{GL}_2(\mathbb{Z}_\ell)$.
\begin{enumerate}
\item Suppose that $E \subseteq K$: then $G_\infty$ is contained in $\prod_\ell C_\ell$, and the index $\left[\prod_\ell C_\ell : G_\infty \right]$ does not exceed $3[K:\mathbb{Q}]$. The equality $G_{\ell^\infty}=C_\ell$ holds for every prime $\ell$ unramified in $K$ and such that $A$ has good reduction at all places of $K$ of characteristic $\ell$.
\item Suppose that $E \not \subseteq K$: then $G_\infty$ is contained in $\prod_\ell N(C_\ell)$ but not in $\prod_\ell C_\ell$, and the index $\left[\prod_\ell N(C_\ell) : G_\infty \right]$ is not finite. The intersection $H_\infty=G_\infty \cap \prod_\ell C_\ell$ has index 2 in $G_\infty$, and the index $\left[\prod_\ell C_\ell : H_\infty \right]$ does not exceed $6[K:\mathbb{Q}]$. The equality $G_{\ell^\infty}=N(C_\ell)$ holds for every prime $\ell$ unramified in $K \cdot E$ and such that $A$ has good reduction at all places of $K$ of characteristic $\ell$.
\end{enumerate}
Finally, the constants 3 and 6 appearing in parts (1) and (2) respectively can be replaced by 1 and 2 if we further assume that the $j$-invariant of $A$ is neither 0 nor 1728.
\end{theorem}

As a by-product of the proof of theorem \ref{thm_Finale_Intro} we also obtain the following proposition, which slightly strengthens a result first proved by Banaszak, Gajda and Kraso\'n (\cite[Theorem A]{MR1971250}) by removing both the assumption that the CM type of $A$ is nondegenerate and the hypothesis that $\ell$ is completely split in $K$.

\begin{proposition}\label{prop_BGK_Intro}{(Proposition \ref{prop_BGK})}
Let $A/K$ be an absolutely simple abelian variety admitting complex multiplication (over $K$) by the CM field $E$, and let $\ell$ be a prime unramified in $E$. Let $E^*$ be the reflex field of $E$ and suppose that $A$ has good reduction at all places of $K$ of characteristic $\ell$.
\begin{itemize}
\item The index $[\MT(A)(\mathbb{F}_\ell):G_\ell \cap \MT(A)(\mathbb{F}_\ell)]$ divides $[\fieldExtension] \cdot |F|$.
\item If $\ell$ is also unramified in $K$, then $[\MT(A)(\mathbb{F}_\ell):G_\ell \cap \MT(A)(\mathbb{F}_\ell)]$ divides $|F|$.
\end{itemize}
\end{proposition}

Let us conclude this introduction by giving a brief overview of the material in the paper.

In section \ref{sect_PrelAlgTori} we recall some fundamental notions about algebraic tori over $\mathbb{Q}$ and their $\mathbb{Z}_\ell$-points; this part also includes a brief account of the theory of abelian varieties of CM type and of their Mumford-Tate groups. In section \ref{sect_CohomologyAndTori} we apply cohomological machinery to study the map induced on $\mathbb{Z}_\ell$-points by algebraic maps between $\mathbb{Q}$-tori with good reduction at $\ell$. With more effort, the method could also give results in the bad reduction setting, but the argument would become quite cumbersome and the result would not be very satisfactory for our purposes. To remedy this situation, in section \ref{sect_Geometry} we treat the case of arbitrary reduction through a purely geometric argument inspired by \cite{MR2947946}; it should be pointed out, however, that -- in the good reduction setting -- the cohomological approach gives much sharper bounds. In section \ref{sect_FinalProofs} we recall a form of the Fundamental Theorem of Complex Multiplication, which gives a complete description of the Galois representations attached to $A$, and apply it to deduce theorem \ref{thm_Finale_Intro}. In section \ref{sect_NondegenerateCase} we give bounds on the order of $\MT(A)(\mathbb{Z}/\ell^n\mathbb{Z})$ under the assumption that $A$ is of nondegenerate type, i.e. that $\operatorname{rank} \MT(A)=\dim A+1$. Finally, in the short section \ref{sect_Examples} we give a simple example that shows that the optimal bound on $\ell^{n \operatorname{rank} \MT(A)} \bigm/ [K(A[\ell^n]):K]$ grows at least exponentially fast in $g$, so that our bounds are not too far from the truth.

\medskip

\noindent\textbf{Acknowledgments.} I thank Nicolas Ratazzi for his unending support and invaluable advice. I am also grateful to Jacob Tsimerman for his help with understanding parts of his paper \cite{MR2947946}.

\section{Preliminaries on algebraic tori}\label{sect_PrelAlgTori}
Recall that over a perfect field $k$ there is an equivalence of categories between algebraic tori and finitely generated, torsion-free, continuous $\abGal{k}$-modules: if $T$ is a $k$-torus, the corresponding $\abGal{k}$-module is the group of characters $\hat{T}=\operatorname{Hom}\left( T_{\overline{k}} , \mathbb{G}_{m,\overline{k}} \right)$. Also recall that this construction extends to an equivalence between finitely generated, continuous $\abGal{k}$-modules and $k$-group schemes of multiplicative type; we will make use of this fact to study the kernel of the reflex norm.
We now introduce a family of $\mathbb{Q}$-algebraic tori that will be especially relevant for us:
\begin{definition}
If $E$ is any number field we set $T_E = \operatorname{Res}_{E/\mathbb{Q}} (\mathbb{G}_{m,\mathbb{Q}})$.
\end{definition}
The torus $T_E$ is of rank $[E:\mathbb{Q}]$, and it admits a very simple description in terms of characters: it is the $\mathbb{Q}$-torus that corresponds to the free module over the set $\operatorname{Hom}(E,\overline{\mathbb{Q}})$, endowed with its natural (right) $\abGal{\mathbb{Q}}$-action.

\begin{proposition}\label{prop_GoodReductionNumberFieldTori} Let $E$ be a number field. The torus $T_E$ has good reduction at all the primes not dividing $\operatorname{disc}(E)$.
\end{proposition}

\begin{proof}
By the Galois criterion (\cite[Proposition 1.1]{MR1129520}), $T_E$ has good reduction at $\ell$ if and only if the inertia group at (a place of $\overline{\mathbb{Q}}$ over) $\ell$ acts trivially on $\widehat{T_E}$. In the present case $\widehat{T_E}$ is the free module over $\operatorname{Hom}(E,\overline{\mathbb{Q}})$, so if we let $L$ be the Galois closure of $E$ in $\overline{\mathbb{Q}}$ the action of $\abGal{\mathbb{Q}}$ on $\widehat{T_E}$ factors through its finite quotient $\operatorname{Gal}(L/\mathbb{Q})$. Now if a prime $\ell$ is unramified in $E$ it is also unramified in $L$, hence the inertia at $\ell$ has trivial image in $\operatorname{Gal}(L/\mathbb{Q})$ and $T_E$ has good reduction at $\ell$, as claimed.
\end{proof}

\subsection{Points of tori with values in $\mathbb{Z}_\ell$ and $\mathbb{Z}/\ell^n \mathbb{Z}$}
We briefly discuss the various possible definitions for the group of $\mathbb{Z}_\ell$-valued points of a $\mathbb{Q}_\ell$-torus; our main reference for this section is \cite[§2]{MR608640}. Let $T$ be a $\mathbb{Q}_\ell$-torus, not necessarily having good reduction over $\mathbb{F}_\ell$. We fix a finite Galois extension $L$ of $\mathbb{Q}_\ell$ that splits $T$, and we regard $\hat{T}$ as a $\Gamma$-module, where $\Gamma:=\operatorname{Gal}(L/\mathbb{Q}_\ell)$. Also notice that a character $\chi \in \hat{T}$ can in particular be considered as a homomorphism $\chi : T(L) \to L^\times$.

\begin{definition}\label{def_ZlPoints}
Following Ono (cf.~\cite[§2]{MR0124326}), we define $T(\mathbb{Z}_\ell)$ to be $\operatorname{Hom}_\Gamma\left(\hat{T}, \mathcal{O}_L^\times \right)$, the group of $\Gamma$-equivariant morphisms (of abelian groups) of $\hat{T}$ in $\mathcal{O}_L^\times$. Equivalently, $T(\mathbb{Z}_\ell)$ is the maximal compact subgroup of $T(\mathbb{Q}_\ell)$.
\end{definition}

If furthermore we suppose that $T$ has good reduction, then it is known (\cite[Theorem 2 on p.109]{MR1634406}) that there exists a $\mathbb{Z}_\ell$-model $\mathcal{T}$ of $T$ (that is, a commutative smooth group scheme over $\operatorname{Spec}(\mathbb{Z}_\ell)$ whose generic fiber is $T$). As pointed out in \cite[Remark 2.2]{MR608640}, in this case the $\mathbb{Z}_\ell$-points of $T$ in the sense of Ono agree with the $\mathbb{Z}_\ell$-valued points of $\mathcal{T}$, so that we are free to use whichever definition we find more convenient.
When a smooth model $\mathcal{T}$ exists we can also give the following definition:
\begin{definition}
If $T$ has good reduction, the $\mathbb{Z}/\ell^n\mathbb{Z}$-points of $T$ are the $\mathbb{Z}/\ell^n\mathbb{Z}$-valued points of its smooth $\mathbb{Z}_\ell$-model $\mathcal{T}$. 
\end{definition}

We still need to discuss the meaning of $T(\mathbb{Z}/\ell^n\mathbb{Z})$ when $T$ does \textit{not} have good reduction. The construction in this case is again due to Ono. For $n \geq 0$, we define subgroups of $T(\mathbb{Q}_\ell)$ by the rule
\[
T(1+\ell^n\mathbb{Z}_\ell) = \left\{ x\in T(\mathbb{Q}_\ell) \bigm\vert v_\ell(\chi(x)- 1) \geq n \quad \forall \chi \in \hat{T} \right\}.
\]
We simply write $T(\mathbb{Z}_\ell)$ for the group corresponding to $n=0$: it can be easily checked that this definition agrees with our previous ones. We can now set $\displaystyle 
T(\mathbb{Z}/\ell^n\mathbb{Z})= \frac{T(\mathbb{Z}_\ell)}{T(1+\ell^n \mathbb{Z}_\ell)}$; once again, when $T$ has a smooth $\mathbb{Z}_\ell$-model $\mathcal{T}$, the group $T(\mathbb{Z}/\ell^n\mathbb{Z})$ agrees with $\mathcal{T}(\mathbb{Z}/\ell^n\mathbb{Z})$.
Finally, when $T$ is a $\mathbb{Q}$-torus we define $T(\mathbb{Z}/\ell^n\mathbb{Z})$ to be the group of $\mathbb{Z}/\ell^n\mathbb{Z}$-points of $T \otimes \mathbb{Q}_\ell$. We conclude this discussion with the following well-known lemma:

\begin{lemma}\label{lemma_Nicolas} Let $T/\mathbb{Q}_\ell$ have good reduction. For every positive integer $n$ we have
\[
(1-1/\ell)^{\dim T} \ell^{n \dim T} \leq \left| T(\mathbb{Z}/\ell^n\mathbb{Z}) \right| \leq (1+1/\ell)^{\dim T} \ell^{n \dim T}.
\]
\end{lemma}

\begin{proof}
A combination of Hensel's lemma and \cite[Theorem 2 on p.104]{MR1634406}; for further details, we refer the reader to \cite[Lemme 2.1 and Proposition 2.2]{MR2862374}.
\end{proof}

\subsection{CM types and reflex norm}

We briefly recall the notions of CM type, of reflex type, and of reflex norm; we refer the reader to \cite[§3]{MR608640} for further details. Let $E$ be a CM field of degree $2g$ and $\tilde{E}$ be its Galois closure in $\overline{\mathbb{Q}}$, and write $G,H$ for the Galois groups $\operatorname{Gal}(\tilde{E}/\mathbb{Q})$ and $\operatorname{Gal}(\tilde{E}/E)$ respectively. We denote by $\tau$ the complex conjugation of $\mathbb{C}$, or any of its restrictions, and we take the convention that the set $\operatorname{Hom}(E,\overline{\mathbb{Q}})$ be identified with the coset space $H\backslash G$.

\begin{lemma}\label{lemma_CMGaloisClosure}
The degree $[\tilde{E}:\mathbb{Q}]$ divides $2^g g!$.
\end{lemma}
\begin{proof}
Let $E_0$ be the maximal totally real subfield of $E$ and $a \in E_0$ be such that $E=E_0(\sqrt{a})$. Let $\tilde{E_0}$ be the Galois closure of $E_0$ and $a_1=a,\ldots,a_k \in \tilde{E_0}$ be the conjugates of $a$ over $\mathbb{Q}$, where $k \leq [E_0:\mathbb{Q}]=g$. It is clear that $\tilde{E}$ is generated over $\tilde{E_0}$ by $\sqrt{a_1}, \ldots, \sqrt{a_k}$, so $[\tilde{E}:\mathbb{Q}]$ divides $[\tilde{E_0}:\mathbb{Q}] \cdot 2^k$. As $[\tilde{E_0}:\mathbb{Q}] \bigm\vert g!$ and $k \leq g$ the lemma follows.
\end{proof}

\begin{definition}
A CM-type for the CM field $E$ is a subset $S$ of $H \backslash G$ such that $S \cap \tau(S) = \emptyset$ and $H \backslash G = S \cup \tau(S)$.
\end{definition}

Let $S$ be a CM type for $E$ and $\tilde{S}$ be the inverse image of $S$ in $G$, i.e. $\tilde{S}=\left\{g \in G \bigm\vert Hg \in S \right\}$. We set
$
H' = \left\{ g \in G \bigm\vert \tilde{S}g=\tilde{S} \right\}
$
and let $E^*$ be the fixed field of $H'$; we then set $\tilde{R} = \left\{ s^{-1} \bigm\vert s \in \tilde{S} \right\}$ and let $R$ be the image of $\tilde{R}$ in $H' \backslash G \cong \operatorname{Hom}\left(E^*,\overline{\mathbb{Q}}\right)$. It is not hard to check that $R$ is a CM type for $E^*$.
\begin{definition} The pair $(E^*,R)$ is called the \textbf{reflex type} of $(E,S)$.
\end{definition}
Finally, a CM type $(E,S)$ is called simple if the equality
\[
H=\left\{ g \in G \bigm\vert g\tilde{S}=\tilde{S} \right\}
\]
holds. We are now ready to define the reflex norm:
\begin{definition}\label{def_ReflexNorm}
Let $(E,S)$ be a CM type, $\tilde{E}$ the Galois closure of $E/\mathbb{Q}$ and $(E^*,R)$ the reflex type of $(E,S)$. The \textbf{reflex norm} associated with $(E,S)$ is the $\mathbb{Q}$-morphism 
\[
\Phi_{(E,S)}: T_{E^*} \to T_E
\]
of algebraic tori given on characters by
\[
\begin{array}{ccccc} \Phi_{(E,S)}^* & : & \widehat{T_E} & \to & \widehat{T_E^*} \\
& & [g] & \mapsto & \sum_{r \in R} [rg],
\end{array}
\]
where $[g]$ (resp.~$[rg]$) is the embedding of $E$ (resp.~$E^*$) in $\overline{\mathbb{Q}}$ induced by the automorphism $g \in \operatorname{Gal}(\tilde{E}/\mathbb{Q})$ (resp.~$rg \in \operatorname{Gal}(\tilde{E}/\mathbb{Q})$).
\end{definition}

\subsection{The Mumford-Tate group}\label{sect_MT}
Our interest in the reflex norm stems from the fact that it allows us to define the Mumford-Tate group of a CM abelian variety rather directly. Before doing so, however, we need to recall how one associates a CM type with a CM abelian variety.

Let $A/K$ be an absolutely simple abelian variety, admitting complex multiplication (over $\overline{K}$) by the field $E$. The tangent space at the identity of $A_{\overline{K}}$ is a $\overline{K}$-module and an $E$-module, and the two actions are compatible: it follows that this tangent space is a $(E \otimes \overline{K})$-bimodule, so it decomposes as $T_{\operatorname{id}} A_{\overline{K}} \cong \prod_{\varphi \in S} \overline{K}_{\varphi}$, where $\overline{K}_{\varphi}$ is a 1-dimensional $\overline{K}$-vector space on which $E$ acts through the embedding $\varphi:E \hookrightarrow \overline{K}$. The set $S$ of embeddings that appear in this decomposition can be shown to be a CM type for $E$, and in this case we say that $A$ admits complex multiplication by the CM type $(E,S)$. When furthermore we have $\operatorname{End}_K(A)=\operatorname{End}_{\overline{K}}(A)$ we say that $A$ admits complex multiplication by $(E,S)$ \textit{over $K$}.

\begin{definition}\label{def_MT}
Let $A/K$ be an absolutely simple abelian variety admitting complex multiplication (over $\overline{K}$) by the CM type $(E,S)$, and let $(E^*,R)$ be the reflex type. We define the Mumford-Tate torus $\MT(A)$ to be the image of the reflex norm $\Phi_{(E,S)} :T_{E^*} \to T_E$.
\end{definition}

\begin{remark}
The Mumford-Tate group of $A$ is in fact a purely geometric object -- it can described in terms of the Hodge structure associated with the complex abelian variety $A_\mathbb{C}$. In particular, it is insensitive to extensions of the base field $K$.
\end{remark}

\begin{remark}
It is known that the rank of $\MT(A)$ is at most $g+1$. When equality holds, the CM type is said to be nondegenerate, and the Mumford-Tate group has a very simple description in terms of $E$: if $\tau$ denotes complex conjugation on $E$, for any $\mathbb{Q}$-algebra $B$ the $B$-points of $\MT(A)$ are given by
\[
\MT(A)(B)=\left\{ x \in \left(E \otimes_{\mathbb{Q}} B\right)^\times \bigm\vert x \tau(x) \in B^\times \right\}.
\]

For all these facts see for example \cite{MR608640}, Proposition 3.3 and the remarks following it.
\end{remark}

\subsection{The group of connected components of $\ker \Phi_{(E,S)}$}
An object which will be crucial to our study is the kernel of the reflex norm $\Phi_{(E,S)}$: in this short subsection we establish a bound on the order of its group of components. The bound is ultimately a consequence of Hadamard's inequality, which is the main tool used to establish the following lemma:

\begin{lemma}\label{lemma_Det01} Let $A$ be a $n \times n$ integral matrix all of whose entries are in $\left\{0,1\right\}$. The following inequality holds:
\[
\left|\det A\right| \leq \lfloor 2^{-n} (n+1)^{(n+1)/2} \rfloor.
\]
\end{lemma}

\begin{proof}
Consider the matrix
\[
B(A)=
\left(
\begin{array}{c|c}
  1 & 1 \cdots 1 \\ \hline
  0 & \raisebox{-15pt}{{\LARGE\mbox{{$2A$}}}} \\[-4ex]
  \vdots & \\[-0.5ex]
  0 &
\end{array}
\right).
\]

It is clear by definition that $\det B(A)=2^n\det(A)$. Consider the matrix $H(A)$ obtained from $B(A)$ by subtracting the first row to each of the others. Clearly $H(A)$ and $B(A)$ have the same determinant, and furthermore all the entries of $H(A)$ are in $\left\{\pm 1\right\}$. In particular, the $L^2$-norm of every row of $H(A)$ is $\sqrt{n+1}$, so Hadamard's inequality implies
\[
\left|\det A \right| = 2^{-n} \left|\det B(A) \right| = 2^{-n} \left|\det H(A) \right| \leq 2^{-n} (n+1)^{(n+1)/2}.
\]
The claim then follows from the fact that $\det(A)$ is an integer.
\end{proof}

\begin{lemma}\label{lemma_OrderGroupConnComponents}
Let $T:\mathbb{Z}^n \to \mathbb{Z}^m$ be a linear map, represented in the standard bases by a matrix $A$ all of whose entries are in $\left\{0,1\right\}$. Let $Y$ be the image of $T$, denote by $k$ the rank of $Y$, and let $Z$ be given by
\[
Z=\left\{z \in \mathbb{Z}^m \bigm\vert \exists q \in \mathbb{Z} \text{ such that } qz \text{ belongs to } Y\right\}.
\]
The quotient $Z/Y$, which is isomorphic to the torsion part of $\mathbb{Z}^m/Y$, has order at most $\lfloor 2^{-k} (k+1)^{(k+1)/2}  \rfloor$.
\end{lemma}

\begin{proof}
The order of $Z/Y$ is given by
\[
\displaystyle \operatorname{gcd} \left\{ \det(A_k) \bigm\vert A_k \text{ is a minor of } A \text{ of size } k \right\}.
\] Lemma \ref{lemma_Det01} ensures that the determinant of every minor of size $k$ does not exceed $\lfloor 2^{-k} (k+1)^{(k+1)/2} \rfloor$, and the lemma follows.
\end{proof}

\begin{proposition}\label{prop_ConnCompMT}
Let $C$ be the group of multiplicative type defined by the exact sequence
\[
1 \to C \to T_{E^*} \xrightarrow{\Phi_{(E,S)}} \MT(A) \to 1
\]
and let $\hat{C}$ be its character group. Suppose $\MT(A)$ has rank $r$. The torsion subgroup of $\hat{C}$ has order at most $\lfloor 2^{-r} (r+1)^{(r+1)/2}  \rfloor$.
\end{proposition}
\begin{proof}
Let $Y$ be the image of $\Phi_{(E,S)}^* : \hat{T}_E \to \hat{T}_{E^*}$ and
\[
Z=\left\{ \chi \in \hat{T}_{E^*} \bigm\vert \exists n \in \mathbb{Z} \text{ such that } n\chi \in Y \right\}.
\]

The torsion subgroup of $\hat{C}$ is isomorphic to $Z/Y$. Moreover, it is apparent from definition \ref{def_ReflexNorm} that the matrix representing $\Phi_{(E,S)}^*$ in the natural bases of $\widehat{T_{E^*}}, \widehat{T_E}$ has entries in $\left\{0,1\right\}$, so the proposition follows from lemma \ref{lemma_OrderGroupConnComponents}.
\end{proof}

\section{Cohomology and integral points of tori}\label{sect_CohomologyAndTori}
The purpose of this section is to study the map induced on $\mathbb{Z}_\ell$-points by a surjection of tori over $\mathbb{Q}_\ell$. More precisely, we let $T \stackrel{\beta}{\longrightarrow} T'' \to 1$ be a surjection of $\mathbb{Q}_\ell$-algebraic tori, and we assume that $T$ has good reduction. We let $T'$ be the kernel of $\beta$, which is in general just a group of multiplicative type (and not necessarily a torus), and write $F$ for the torsion subgroup of its character group $\hat{T'}$. We also denote by $a$ the rank of $T'$, so that we have an isomorphism of abelian groups $\hat{T'}/F \cong \mathbb{Z}^a$. Finally, we fix a finite \textit{unramified} Galois extension $L$ of $\mathbb{Q}_\ell$ that splits $T$, and we let $\Gamma$ denote the Galois group of $L$ over $\mathbb{Q}_\ell$. It is also useful to introduce the following notation:

\smallskip

\noindent\textbf{Notation.} If $n$ is any integer and $\ell$ is a prime we write $|n|_\ell$ for $\ell^{-v_\ell(n)}$. When $M$ is a finite group we also write $|M|_\ell$ for $\ell^{-v_\ell(|M|)}$.

\smallskip

With this notation we shall show:

\begin{proposition}\label{prop_UnramifiedBound}
The cokernel of $T(\mathbb{Z}_\ell) \stackrel{\beta}{\longrightarrow} T''(\mathbb{Z}_\ell)$ has order dividing $|F| \cdot |F|_\ell^{-[L:\mathbb{Q}_\ell]}$.
\end{proposition}

The proof is given below in §\ref{sect_ProofUnramifiedBound}, and relies mainly on the basic tools of Galois cohomology, together with the following classical theorem of Nakayama (cf. for example \cite[§2, Theorem 32]{MR0347778}):
\begin{theorem}\label{thm_NakayamaVanishing}
Let $A$ and $B$ be modules over the finite group $G$. Assume that $A$ is cohomologically trivial. In order for $\operatorname{Hom}(B,A)$ to be cohomologically trivial it is necessary and sufficient that $\operatorname{Ext}^1(B,A)$ be cohomologically trivial. In particular, if $B$ is $\mathbb{Z}$-free, then $\operatorname{Hom}(B,A)$ is cohomologically trivial.
\end{theorem}

\subsection{Preliminaries on $p$-adic fields}
The following two lemmas are certainly well-known, but for lack of an easily accessible reference we prefer to include a short proof.

\begin{lemma}
Let $L$ be a finite extension of $\mathbb{Q}_\ell$ with ring of integers $\mathcal{O}_L$, and let $n$ be a positive integer. The quotient $\mathcal{O}_L/\mathcal{O}_L^{\times n}$ has order dividing $n \cdot |n|_\ell^{-[L:\mathbb{Q}_\ell]}$.
\end{lemma}

\begin{proof}
We regard all the involved groups as $\mathbb{Z}/n\mathbb{Z}$-modules with trivial action, and denote by $h_n$ the associated Herbrand quotient, that is to say for every finite $\mathbb{Z}/n\mathbb{Z}$-module $M$ we set
\[
\displaystyle h_n(M) := \frac{|\hat{H}^0(\mathbb{Z}/n\mathbb{Z},M)|}{|H^1(\mathbb{Z}/n\mathbb{Z},M)|}.
\]
As $\mathcal{O}_L^1$, the subgroup of principal units of $\mathcal{O}_L$, has finite index in $\mathcal{O}_L^\times$ (and the Herbrand quotient is invariant by passage to finite-index subgroups), we have $h_n(\mathcal{O}_L^\times)=h_n(\mathcal{O}_L^1)$.

On the other hand, $\mathcal{O}_L^1$ contains a subgroup of finite index that is isomorphic to $\mathcal{O}_L$ (\cite[Chapitre XIV, prop. 10]{MR0354618}), so $h_n(\mathcal{O}_L^\times)=h_n(\mathcal{O}_L^1)=h_n(\mathcal{O}_L)$. Furthermore, $H^1(\mathbb{Z}/n\mathbb{Z},\mathcal{O}_L^\times)=\operatorname{Hom}\left(\mathbb{Z}/n\mathbb{Z},\mathcal{O}_L^\times \right)=\mathcal{O}_L^\times[n]$ has order dividing $n$, while $H^1(\mathbb{Z}/n\mathbb{Z},\mathcal{O}_L)=\mathcal{O}_L[n]=0$. The lemma then follows easily because the quantity $\left|\frac{\mathcal{O}_L^\times}{\mathcal{O}_L^{\times n}}\right|=h^1(\mathbb{Z}/n\mathbb{Z},\mathcal{O}_L^\times) \cdot h_n\left( \mathcal{O}_L^\times \right)$ divides
\[
n \cdot h_n(\mathcal{O}_L) = n \frac{\left|\mathcal{O}_L/n\mathcal{O}_L \right|}{h^1(\mathbb{Z}/n\mathbb{Z},\mathcal{O}_L)}=n \cdot |n|_\ell^{-[L:\mathbb{Q}_\ell]}.
\]
\end{proof}

\begin{lemma}\label{lemma_ExtBound}
Let $F$ be a finite abelian group and $L$ be a finite extension of $\mathbb{Q}_\ell$. Then $|\operatorname{Ext}^1(F,\mathcal{O}_L^\times)|$ divides $|F| \cdot |F|_\ell^{-[L:\mathbb{Q}_\ell]}.$
\end{lemma}

\begin{proof}
Writing $F$ as $\displaystyle \bigoplus_{i} \frac{\mathbb{Z}}{d_i \mathbb{Z}}$ we have \[\operatorname{Ext}^1\left(F,\mathcal{O}_L^\times\right) \cong \prod_i \operatorname{Ext}^1\left(\frac{\mathbb{Z}}{d_i \mathbb{Z}}, \mathcal{O}_L^\times \right) \cong \prod_i \frac{\mathcal{O}_L^\times}{\mathcal{O}_L^{\times d_i}}.\] The result follows from the previous lemma.
\end{proof}

\subsection{Proof of proposition \ref{prop_UnramifiedBound}}\label{sect_ProofUnramifiedBound}
Note that -- since $L/\mathbb{Q}_\ell$ is unramified -- the group $\mathcal{O}_L^\times$ is a cohomologically trivial $\Gamma$-module (cf. for example \cite[Prop. 7.1.2 (i)]{neukirch2013cohomology}). As $\hat{T}$ and $\hat{T''}$ are free abelian groups, Nakayama's theorem implies in particular that $\operatorname{Hom}\left( \hat{T}, \mathcal{O}_L^\times \right)$ and $\operatorname{Hom}\left( \hat{T''}, \mathcal{O}_L^\times \right)$ are cohomologically trivial $\Gamma$-modules. We will make extensive use of this fact. The character groups of $T,T',T''$ fit into an exact sequence
\[
0 \to \hat{T''} \to \hat{T} \to \hat{T'} \to 0;
\]
applying the functor $\operatorname{Hom}\left(-, \mathcal{O}_L^\times \right)$ gives another exact sequence
\[
\begin{aligned}
0 \to \operatorname{Hom}\left(\hat{T'},\mathcal{O}_L^\times\right) & \to \operatorname{Hom}\left(\hat{T},\mathcal{O}_L^\times\right) \\ &\to \operatorname{Hom}\left(\hat{T''},\mathcal{O}_L^\times\right)  \to \operatorname{Ext}^1\left(\hat{T'}, \mathcal{O}_L^\times \right) \to 0,
\end{aligned}
\]
where the following $\operatorname{Ext}$ term vanishes since $\hat{T}$ is free. If we let
\[
I:=\operatorname{Image}\left( \operatorname{Hom}\left(\hat{T},\mathcal{O}_L^\times\right) \to \operatorname{Hom}\left(\hat{T''},\mathcal{O}_L^\times\right) \right),\]
the previous sequence gives rise to the two exact sequences
\begin{equation}\label{eq_seq1}
0 \to \operatorname{Hom}\left(\hat{T'},\mathcal{O}_L^\times\right) \to \operatorname{Hom}\left(\hat{T},\mathcal{O}_L^\times\right) \to I \to 0
\end{equation}
and
\begin{equation}\label{eq_seq2}
0 \to I \to \operatorname{Hom}\left(\hat{T''},\mathcal{O}_L^\times\right) \to \operatorname{Ext}^1\left(\hat{T'}, \mathcal{O}_L^\times\right) \to 0.
\end{equation}

The long exact sequences in Galois cohomology associated with \eqref{eq_seq1} and \eqref{eq_seq2} give
\begin{equation}\label{eq_seq3}
0 \to \operatorname{Hom}_\Gamma\left(\hat{T'},\mathcal{O}_L^\times\right) \to T(\mathbb{Z}_\ell)  \to H^0(\Gamma,I) \to H^1\left(\Gamma, \operatorname{Hom}\left(\hat{T'} ,\mathcal{O}_L^\times\right)\right) \to 0,
\end{equation}
\begin{equation}\label{eq_seq4}
0 \to H^1(\Gamma,I) \to H^2\left( \Gamma, \operatorname{Hom} \left( \hat{T'}, \mathcal{O}_L^\times \right) \right) \to 0,
\end{equation}
and
\begin{equation}\label{eq_seq5}
0 \to H^0(\Gamma,I) \to T''(\mathbb{Z}_\ell) \to H^0\left(\Gamma,\operatorname{Ext}^1\left(\hat{T'}, \mathcal{O}_L^\times\right)\right) \to H^1(\Gamma,I) \to 0,
\end{equation}
where we have used the fact that $\operatorname{Hom}\left( \hat{T}, \mathcal{O}_L^\times \right)$ and $\operatorname{Hom}\left( \hat{T''}, \mathcal{O}_L^\times \right)$ are cohomologically trivial. Also notice that we have an exact sequence of $\Gamma$-modules
\begin{equation}\label{eq_CharacterModule}
0 \to F \to \hat{T'} \to \hat{T'}/F \to 0
\end{equation}
where $\hat{T'}/F \cong \mathbb{Z}^a$ is free. 
We can then apply $\operatorname{Hom}\left(-, \mathcal{O}_L^\times \right)$ to \eqref{eq_CharacterModule} to get
\[
0 \to \operatorname{Hom}\left(\hat{T'}/F, \mathcal{O}_L^\times \right) \to \operatorname{Hom}\left(\hat{T'},\mathcal{O}_L^\times\right) \to \operatorname{Hom}\left(F,\mathcal{O}_L^\times\right) \to 0,
\]
and since $\operatorname{Hom}\left(\hat{T'}/F, \mathcal{O}_L^\times \right)$ is again cohomologically trivial by theorem \ref{thm_NakayamaVanishing} we deduce that for every $n \geq 1$ we have canonical isomorphisms
\begin{equation}\label{eq_CohomologyIsomorphism}
H^n\left(\Gamma, \operatorname{Hom}\left(\hat{T'},\mathcal{O}_L^\times \right) \right) \stackrel{\sim}{\longrightarrow} H^n\left(\Gamma, \operatorname{Hom}\left(F,\mathcal{O}_L^\times \right) \right).
\end{equation}

Straightforward manipulations of sequences \eqref{eq_seq3} and \eqref{eq_seq5} show that
\[
\left| \operatorname{coker} \left( T(\mathbb{Z}_\ell) \to T''(\mathbb{Z}_\ell) \right) \right| = \frac{h^0\left( \Gamma,\operatorname{Ext}^1\left( \hat{T'}, \mathcal{O}_L^\times \right) \right) \cdot h^1\left(\Gamma, \operatorname{Hom}\left(\hat{T'},\mathcal{O}_L^\times\right) \right)}{h^1(\Gamma,I)}.
\]
For the sake of notational simplicity set $M=\operatorname{Hom}\left(F,\mathcal{O}_L^\times\right)$. Using \eqref{eq_seq4} and \eqref{eq_CohomologyIsomorphism} we arrive at
\begin{equation}\label{eq_AlmostFinal}
\left| \operatorname{coker} \left( T(\mathbb{Z}_\ell) \to T''(\mathbb{Z}_\ell) \right) \right|  = \frac{h^0\left( \Gamma,\operatorname{Ext}^1\left( \hat{T'}, \mathcal{O}_L^\times \right) \right) \cdot h^1\left(\Gamma, M \right)}{h^2(\Gamma,M)}.
\end{equation}

Observe now that the group $\Gamma$ is cyclic (since it is the Galois group of an unramified extension) and the module $M$ is finite: as it is well-known, the Tate cohomology $\hat{H}^n$ of a cyclic group with values in a finite module is $2$-periodic in $n$. Moreover, the Herbrand quotient $\frac{\left|\hat{H^0}(\Gamma,M)\right|}{\left|\hat{H^1}(\Gamma,M)\right|}$ equals 1 since $M$ is finite, and therefore $h^2(\Gamma,M)=\left| \hat{H}^0(\Gamma,M) \right|=h^1(\Gamma,M)$ (for all these facts cf. for example \cite[§I.7]{neukirch2013cohomology}). Using this equality in \eqref{eq_AlmostFinal} we finally find $\left| \operatorname{coker} \left( T(\mathbb{Z}_\ell) \to T''(\mathbb{Z}_\ell) \right) \right|=h^0\left( \Gamma,\operatorname{Ext}^1\left( \hat{T'}, \mathcal{O}_L^\times \right) \right)$.
Proposition \ref{prop_UnramifiedBound} then follows from the fact that $h^0\left( \Gamma,\operatorname{Ext}^1\left( \hat{T'}, \mathcal{O}_L^\times \right) \right)$ divides $\left| \operatorname{Ext}^1\left( \hat{T'}, \mathcal{O}_L^\times \right) \right| = \left| \operatorname{Ext}^1\left( \mathbb{Z}^a \oplus F, \mathcal{O}_L^\times \right) \right| = \left| \operatorname{Ext}^1\left(F, \mathcal{O}_L^\times \right) \right|$ and from lemma \ref{lemma_ExtBound}.

\section{The cokernel of an isogeny, without the good reduction assumption}\label{sect_Geometry}
Let $T, T'$ be $\mathbb{Q}_\ell$-tori and $\lambda: T \to T'$ be a $\mathbb{Q}_\ell$-isogeny. We do not assume that $T$ or $T'$ has good reduction, and for the purposes of this section we define the $\mathbb{Z}_\ell$-points of a $\mathbb{Q}_\ell$-torus to be the maximal compact subgroup of $T(\mathbb{Q}_\ell)$ (cf. definition \ref{def_ZlPoints}). Our aim is again to bound the order of $\operatorname{coker} \left( T(\mathbb{Z}_\ell) \stackrel{\lambda}{\longrightarrow} T'(\mathbb{Z}_\ell) \right)$, in terms of the degree $m$ of $\lambda$ and of $\dim T=\dim T'=:d$. Cohomological tools could again be used to investigate the problem, but we find that an entirely different approach (through $p$-adic differential geometry) yields simpler and more effective proofs; the method is inspired by \cite{MR2947946}, see especially lemma 4.4 in \textit{op.~cit}.

\begin{proposition}\label{prop_UnconditionalBound}
Let $T,T'$ be $\mathbb{Q}_\ell$-tori of dimension $d$ and $\lambda:T \to T'$ be an isogeny of degree $m$. The order of $\operatorname{coker} \left( T(\mathbb{Z}_\ell) \stackrel{\lambda}{\longrightarrow} T'(\mathbb{Z}_\ell) \right)$ is at most $m^d \cdot |m|_\ell^{-d}$.
\end{proposition}

\begin{proof}
Notice first that $\lambda$ fits into a commutative diagram
\smallskip

\begin{center}
\makebox{\xymatrix{
T'(\mathbb{Z}_\ell) \ar@{-->}[rr]^{\lambda^\vee} \ar[dr]_{[m]} & & T(\mathbb{Z}_\ell) \ar[dl]^{\lambda} \\ & T'(\mathbb{Z}_\ell)  }} 
\end{center}
and therefore it is enough to bound the cokernel of $[m]:T'(\mathbb{Z}_\ell) \to T'(\mathbb{Z}_\ell)$. Fix now a Haar measure $\mu$ on $T'(\mathbb{Q}_\ell)$, normalized in such a way that $\mu(T'(\mathbb{Z}_\ell))=1$. 

Consider the kernel $K$ of $[m]$ (as a subgroup of $T'(\mathbb{Z}_\ell)$, not as a group scheme) and the quotient $S=T'(\mathbb{Z}_\ell)/K$, and note that $\pi : T'(\mathbb{Z}_\ell) \to S$ is a covering map. We denote by $\mu_S$ the measure on $S$ given by $\mu_S(A)=\frac{1}{|K|} \mu \left(\pi^{-1}(A)\right)$: it can also be interpreted as the measure induced on $S$ by the (Haar) volume form of $T'(\mathbb{Z}_\ell)$, which passes to the quotient since it is translation-invariant.
The volume of $S$ (for the measure $\mu_S$) is $\displaystyle \frac{\operatorname{vol} (T'(\mathbb{Z}_\ell))}{|K|}=\frac{1}{|K|}$, and we have an $\ell$-adic analytic map $q:S \to T'(\mathbb{Z}_\ell)$ such that the following diagram commutes:
\begin{center}
\makebox{\xymatrix{
T'(\mathbb{Z}_\ell) \ar[rr]^{[m]} \ar[dr]_{\pi} & & T'(\mathbb{Z}_\ell)  \\ & S\ar@{-->}[ur]_{q}  }}
\end{center}
Clearly $q$ is an $\ell$-adic analytic embedding and we have $\operatorname{Image} q =\operatorname{Image} [m]=:I$. We have the following immediate equality:
\begin{equation}\label{eq_Volumes1}
\operatorname{vol}(I) = \frac{1}{|T'(\mathbb{Z}_\ell)/I|} \operatorname{vol} \left( T'(\mathbb{Z}_\ell) \right)=\frac{1}{|T'(\mathbb{Z}_\ell)/I|}.
\end{equation}
On the other hand, a simple computation in coordinates shows $q^* \mu = |m|_\ell^d \, \mu_S$: we can parametrize a neighbourhood of $g \in T'(\mathbb{Z}_\ell)$ by $x \mapsto g \exp(x)$ (for $x$ varying in some small neighbourhood of 0 in the Lie algebra of $T'(\mathbb{Q}_\ell)$), and composing with $\pi$ this also induces a parametrization of a neighbourhood of $\pi(g) \in S$. In these coordinates the map $q$ is simply multiplication by $m$, so its Jacobian determinant is $|m|_\ell^d$ and the change of variables formula for $\ell$-adic integration gives the required result. This yields
\[
\operatorname{vol}(I)  =\int_{I} d \mu = \int_{q(S)} d \mu = \int_{S} d(q^*  \mu) = \int_S |m|_\ell^d \; d \mu_S =  |m|_\ell^d \; \frac{1}{|K|},
\]
and comparing this equality with equation \eqref{eq_Volumes1} gives
\[
\left|\operatorname{coker} \left(T'(\mathbb{Z}_\ell) \stackrel{[m]}{\longrightarrow} T'(\mathbb{Z}_\ell) \right) \right|=|T'(\mathbb{Z}_\ell)/I|=\frac{1}{\operatorname{vol}(I)} = \frac{|K|}{|m|_\ell^d}.
\]
Finally, it is clear that $|K| \leq \left|T'(\overline{\mathbb{Q}_\ell})[m]\right|=m^d$, and this finishes the proof.
\end{proof}

\section{Description of the Galois representation}\label{sect_FinalProofs}
Let $A/K$ be an absolutely simple $g$-dimensional CM abelian variety admitting complex multiplication (over $K$) by the CM type $(E,S)$.
Let $\tilde{E}$ be the Galois closure of $E$, denote by $(E^*,R)$ the reflex type of $(E,S)$, and let $\ell$ be a prime number. 
It is known that -- since the action of $E$ is defined over $K$ -- the reflex field $E^*$ is contained in $K$ (\cite[Chap. 3, Theorem 1.1]{MR713612}), and by \cite[Corollary 2 to Theorem 5]{MR0236190}, the $\ell$-adic Galois representation attached to $A$ can be viewed as a map
\[
\rho_{\ell^\infty} : \abGal{K} \to \left(\operatorname{End}_K(A) \otimes \mathbb{Z}_\ell \right)^\times \hookrightarrow \left(\mathcal{O}_E \otimes \mathbb{Z}_\ell\right)^\times.
\]
We denote by $G_{\ell^\infty}$ the image of $\rho_{\ell^\infty}$. We now recall the description of $\rho_{\ell^\infty}$ coming from the fundamental theorem of complex multiplication, and refer the reader to \cite[§4]{MR608640} and \cite{MR0236190} for further details. Let $I_K$ be the group of idèles of $K$. As $\left(\mathcal{O}_E \otimes \mathbb{Z}_\ell\right)^\times$ is commutative, there is a factorization
\begin{center}
\makebox{\xymatrix{
I_K \ar[r] & \abGal{K}^{ab} \ar@{-->}[r] & \left(\mathcal{O}_E \otimes \mathbb{Z}_\ell\right)^\times  \\ & \abGal{K}\ar[ur]_{\rho_{\ell^\infty}} \ar[u]  }}
\end{center}
which (by class field theory) allows us to regard $\rho_{\ell^\infty}$ as a map from $I_K$ to $\left(\mathcal{O}_E \otimes \mathbb{Z}_\ell \right)^\times$. Let us introduce some notation: we write $\mu(E)$ for the group of roots of unity in $E$, and if $v$ is a place of $K$ we write $\mathcal{O}_{K,v}$ for the completion at $v$ of the ring of integers of $K$. If $v$ is furthermore finite we denote by $p_v$ its residual characteristic; we also let $\Omega_K$ be the set of all finite places of $K$. If $F$ is a number field we denote by $F_\ell$ the algebra $F \otimes \mathbb{Q}_\ell$, and for an idèle $a \in I_K$ we write $a_\ell$ for the component of $a$ in $K_\ell\cong \displaystyle \prod_{p_v=\ell} K_v^\times$.
With this notation, the map $\rho_{\ell^\infty}$ is described very precisely by the following theorem:
\begin{theorem}{(\cite[Theorems 6, 10 and 11]{MR0236190})}\label{thm_FundamentalTheoremCM}
There exists a unique continuous homomorphism $\varepsilon : I_K \to E^\times$ such that, for all finite places $v$ of $K$, the group $\varepsilon \left( \mathcal{O}_{K,v}^\times \right)$ is contained in $\mu(E)$, and
\[
\rho_{\ell^\infty}(a) = \varepsilon(a) \Phi_{(E,S)}\left(N_{K_\ell/E_\ell^*} (a_\ell) \right)^{-1}
\]
for all $a \in I_K$. If furthermore $v \in \Omega_K$ is a place of good reduction for $A$, then $\varepsilon\left( \mathcal{O}_{K,v}^\times \right)$ is trivial.
\end{theorem}

We now consider the restriction of $\rho_{\ell^\infty}$ to $K^\times  \cdot \prod_{v \in \Omega_K} \mathcal{O}_{K,v}^\times$: as it is well-known (cf. for example \cite[Proposition 2.3]{MR819231}), this is the group of idèles of $H$, the Hilbert class field of $K$. In terms of Galois groups, this has the effect of restricting $\rho_{\ell^\infty}$ to $\abGal{H} \subseteq \abGal{K}$, so it is clear that $\rho_{\ell^\infty}\left(\abGal{H}\right)$ is a subgroup of $\rho_{\ell^\infty}\left(\abGal{K}\right)$ of index dividing $h(K)$, the class number of $K$. Now as $\rho_{\ell^\infty}$ factors through $\abGal{K}$ we see that $\rho_{\ell^\infty}(K^\times)$ is trivial, so we can just consider the restriction of $\rho_{\ell^\infty}$ to $\prod_{v \in \Omega_K} \mathcal{O}_{K,v}^\times$. We now remark that for an idèle $(a_v) \in \prod_{v \in \Omega_K} \mathcal{O}_{K,v}^\times$ theorem \ref{thm_FundamentalTheoremCM} implies
\[
\varepsilon(a) = \prod_{v \in \Omega_K} \varepsilon\left(a_v \right) = \prod_{\substack{v : A\text{ has bad}\\\text{reduction at v}}} \varepsilon(a_v) \in \mu(E),
\]
whence $J:=\ker \varepsilon \cap \prod_{v \in \Omega_K} \mathcal{O}_{K,v}^\times$ has index dividing $|\mu(E)|$ in $\prod_{v \in \Omega_K} \mathcal{O}_{K,v}^\times$, and likewise the index of $J_\ell:=\ker \varepsilon \cap \prod_{v | \ell} \mathcal{O}_{K,v}^\times$ in $\prod_{v | \ell} \mathcal{O}_{K,v}^\times$ divides $|\mu(E)|$.
Furthermore, since the function $a \mapsto \Phi_{(E,S)}\left(N_{K_\ell/E_\ell^*}(a_\ell) \right)^{-1}$ kills $\mathcal{O}_{K,v}^\times$ when $p_v \neq \ell$, we have $\rho_{\ell^\infty}(J)=\rho_{\ell^\infty}(J_\ell)$. Also notice that, upon restriction to $J_\ell$, the representation $\rho_{\ell^\infty}$ coincides with the map
\[
\begin{array}{cccc}
\displaystyle \varphi_{\ell^\infty} : & \displaystyle \prod_{v | \ell} \mathcal{O}_{K,v}^\times & \to & \left(\mathcal{O}_E \otimes \mathbb{Z}_\ell \right)^\times \\
& a & \mapsto & \Phi_{(E,S)}\left( N_{K_\ell/E^*_\ell}(a)\right)^{-1},
\end{array}
\]
and that if $A$ has good reduction at $v$, then $\rho_{\ell^\infty}$ and $\varphi_{\ell^\infty}$ coincide on all of $\displaystyle \prod_{v | \ell} \mathcal{O}_{K,v}^\times$.
For the sake of notational simplicity let us then set
\[
\mu^*=\begin{cases} |\mu(E)|, \text{ if }A\text{ has bad reduction at}\\ \hspace{35pt} \text{ some place }v \text{ of characteristic }\ell \\ 1, \text{ otherwise}\end{cases}
\]
We have proved:
\begin{proposition}\label{prop_Epsilon}
For all primes $\ell$ the group $G_{\ell^\infty}$ contains $\rho_{\ell^\infty}\left(J_\ell \right)$ as a subgroup of index dividing $|\mu(E)| \cdot h(K)$. We have $\rho_{\ell^\infty}(J_\ell)=\varphi_{\ell^\infty}(J_\ell)$, and if $A$ has good reduction at all places $v$ of characteristic $\ell$ we have $J_\ell= \prod_{v | \ell} \mathcal{O}_{K,v}^\times$. Finally,
\begin{equation}\label{eq_mu}
\left[ \varphi_{\ell^\infty} \left( \prod_{v | \ell} \mathcal{O}_{K,v}^\times\right) : \rho_{\ell^\infty}(J_\ell) \right] \bigm\vert \mu^*.
\end{equation}
\end{proposition}

We can now interpret $\varphi_{\ell^\infty}$ as a map between algebraic tori: indeed, the norm $N_{K/E^*}$ can be seen as a morphism $T_K \to T_{E^*}$, and $\prod_{v | \ell} \mathcal{O}_{K,v}^\times$ is nothing but $T_K(\mathbb{Z}_\ell)$, so the map $\varphi_{\ell^\infty}$ is simply the map induced on $\mathbb{Z}_\ell$-points by
\[
\left(\Phi_{(E,S)}\right)^{-1} \circ N_{K/E^*} : T_K \to \operatorname{MT}(A);
\]
together with the previous proposition, this implies in particular that $\rho_{\ell^\infty}\left( J_\ell \right)=\varphi_{\ell^\infty}\left( J_\ell \right)$ is contained in $\operatorname{MT}(A)(\mathbb{Z}_\ell)$, and that $\varphi_{\ell^\infty}\left( J_\ell \right)$ has index at most $\mu^*$ in $\varphi_{\ell^\infty}\left( T_K(\mathbb{Z}_\ell) \right)$.
We thus want to understand the composition
\[
T_K(\mathbb{Z}_\ell) \xrightarrow{N_{K/E^*}} T_{E^*}(\mathbb{Z}_\ell) \xrightarrow{\psi_\ell} \MT(A)(\mathbb{Z}_\ell),
\]
where for simplicity of notation we write $\psi_\ell$ for the base-change to $\mathbb{Q}_\ell$ of the map $\left(\Phi_{(E,S)}(\cdot)\right)^{-1}$. Even though the extension $K/E^*$ is in general non-abelian, the cokernel of $N_{K/E^*}$ can be understood through class field theory:
\begin{theorem}\label{thm_CokernelOfNorm}{(\cite[Theorem 7 on p.~161]{MR2467155})}
Let $L/M$ be an extension of local fields, and let $L_{ab}$ be the largest abelian subextension of $L/M$. Then we have $N_{L/M} L^\times = N_{L_{ab}/M} \left( L_{ab}^\times \right)$, and the cokernel $\displaystyle \frac{M^\times}{N_{L/M} L^\times}$ has order dividing $[L:M]$.
\end{theorem}

Note that the image of $\psi_\ell$ is open and $\MT(A)(\mathbb{Z}_\ell)$ is compact, so the cokernel of $\psi_\ell: T_{E^*}(\mathbb{Z}_\ell) \stackrel{\psi_\ell}{\longrightarrow}  \MT(A)(\mathbb{Z}_\ell)$ is finite; since furthermore by theorem \ref{thm_CokernelOfNorm} $\displaystyle \left| \frac{T_{E^*}(\mathbb{Z}_\ell)}{N_{K/E^*}(T_K(\mathbb{Z}_\ell))} \right|$ divides $[\fieldExtension]$ we find that 
\begin{equation}\label{eq_Index}
\displaystyle \left[ \MT(A)(\mathbb{Z}_\ell) : \varphi_{\ell^\infty}(T_K(\mathbb{Z}_\ell)) \right] \text{ divides }  [\fieldExtension] \cdot \left| \frac{\MT(A)(\mathbb{Z}_\ell)}{\psi_\ell\left(T_{E^*}(\mathbb{Z}_\ell)\right)} \right|.
\end{equation}

\begin{remark}\label{rmk_NormTotallySplit}
When $\ell$ is unramified in $K$ the local norm $T_K(\mathbb{Z}_\ell) \to T_{E^*}(\mathbb{Z}_\ell)$ is surjective and the factor $[\fieldExtension]$ can be omitted, cf. \cite[Corollary to Proposition 3 of Chapter V]{MR0354618}.
\end{remark}

It is clear that $\psi_\ell=\Phi_{(E,S)}^{-1}$ and $\Phi_{(E,S)}$ have the same cokernel, so ultimately we just need to compute the cokernel of the reflex norm. Denote by $T'$ the kernel of $\Phi_{(E,S)}$ and write $F$ for the torsion of its character group $\widehat{T'}$. By proposition \ref{prop_ConnCompMT} we have $|F| \leq \lfloor 2^{-r} (r+1)^{(r+1)/2} \rfloor$, where $r=\dim \operatorname{Im} \Phi_{(E,S)}^* = \operatorname{rk} \MT(A)$ does not exceed $g+1$.
Set now $T=T_{E^*} \otimes \mathbb{Q}_\ell$ and $T''=\MT(A) \otimes \mathbb{Q}_\ell$, and let $L$ be one of the fields appearing in the decomposition of $\tilde{E} \otimes \mathbb{Q}_\ell$ as a direct sum of fields: $L/\mathbb{Q}_\ell$ is then a finite Galois extension that splits $T$ (recall that $\tilde{E}$ is Galois and contains $E^*$). If $\ell$ is unramified in $E$ (hence in $\tilde{E}$) the extension $L/\mathbb{Q}_\ell$ is itself unramified, so $T$ has good reduction over $\mathbb{Q}_\ell$; furthermore, $[L:\mathbb{Q}_\ell] \bigm\vert [\tilde{E} : \mathbb{Q}] \bigm\vert 2^g \cdot g!$ (cf. lemma \ref{lemma_CMGaloisClosure}).

Applying proposition \ref{prop_UnramifiedBound} to the surjection of algebraic tori $T \xrightarrow{\Phi_{(E,S)}} T''$ we find that
\begin{equation}\label{eq_divisibility}
\left| \operatorname{coker} \left( T_{E^*}(\mathbb{Z}_\ell) \xrightarrow{\Phi_{(E,S)}} \MT(A)(\mathbb{Z}_\ell) \right) \right| \text{ divides } |F| \cdot |F|_\ell^{-[L:\mathbb{Q}_\ell]},
\end{equation}
and the right hand side in turn divides $|F| \cdot |F|_\ell^{-2^g g!}$; we have thus almost completely established the following result:

\begin{theorem}\label{thm_Finale}{(Theorem \ref{thm_Finale_Intro})}
Let $A/K$ be an absolutely simple abelian variety of dimension $g$ admitting complex multiplication over $K$ by the CM type $(E,S)$, and let $\ell$ be a prime number. If $A$ has bad reduction at a place of $K$ dividing $\ell$ let $\mu^*=|\mu(E)|$, the number of roots of unity in $E$; if on the contrary $A$ has good reduction at all places of $K$ of characteristic $\ell$ set $\mu^*=1$.
Denote by $r$ the rank of $\MT(A)$ and by $F$ the group of connected components of the kernel of the reflex norm $T_{E^*} \to T_{E}$, where $E^*$ is the reflex field of $E$. Then:

\begin{enumerate}
\item[(1)] The index $\left[G_{\ell^\infty} : G_{\ell^\infty} \cap \MT(A)(\mathbb{Z}_\ell) \right]$ does not exceed $|\mu(E)| \cdot h(K)$, where $h(K)$ is the class number of $K$.
\item[(2)] We have $\left[\MT(A)(\mathbb{Z}_\ell): G_{\ell^\infty} \cap \MT(A)(\mathbb{Z}_\ell) \right] \leq \mu^* \cdot [\fieldExtension] \cdot |F|^{2r}$.
\item[(3)] If $\ell$ is unramified in $E$ and does not divide $|F|$, then the index $\left[\MT(A)(\mathbb{Z}_\ell): G_{\ell^\infty} \cap \MT(A)(\mathbb{Z}_\ell)\right]$ divides $\mu^* \cdot [\fieldExtension] \cdot |F|$. If $\ell$ is also unramified in $K$, the bound can be improved to $\mu^* \cdot |F|$.
\end{enumerate}
Finally we have $r\leq g+1$ and $|F| \leq f(r) \leq f(g+1)$, where \[\displaystyle 
f(x)=\left\lfloor 2 \left( \frac{x+1}{4} \right)^{(x+1)/2} \right\rfloor.\]
\end{theorem}

\begin{proof}
We have already proved (1): the intersection $G_{\ell^\infty} \cap \MT(A)(\mathbb{Z}_\ell)$ contains $\varphi_{\ell^\infty}(J_\ell)=\rho_{\ell^\infty}(J_\ell)$, and by proposition \ref{prop_Epsilon} the group $\varphi_{\ell^\infty}(J_\ell)$ has index at most $|\mu(E)| \cdot h(K)$ in $G_{\ell^\infty}$. As for part (2), the exact sequence
\[
1 \to T' \to T_{E^*} \otimes \mathbb{Q}_\ell \to \MT(A) \otimes \mathbb{Q}_\ell \to 1
\]
induces, by quotienting out by $(T')^0$ (the connected component of the identity of $T'$), the exact sequence
\[
1 \to \mathcal{F} \to \frac{T_{E^*} \otimes \mathbb{Q}_\ell}{(T')^0} \stackrel{\tau_\ell}{\longrightarrow} \MT(A) \otimes \mathbb{Q}_\ell \to 1,
\]
% --------------------------------------------------------
% Here is the relevant commutative diagram. The second line is the existing sequence, the first line is that one obtains by quotienting out.
%\xymatrix{
%1  & \text{tors}(\hat{C}) \ar[l] \ar@{^{(}->}[d] & & \ker \pi_2 \ar[ll] \ar@{^{(}->}[dl] & \widehat{\MT(A)} \ar[l] \ar[dl]_{=} & 1 \ar[l] \\
%1  & \hat{C} \ar[l] \ar[d]_{\pi_1} & \widehat{T_{E^*}} \ar[l] \ar[dl]^{\pi_2} & \widehat{\MT(A)} \ar[l] & 1 \ar[l] \\ & \hat{C}/\text{tors}(\hat{C}) \ar[d] \ar[dl]  \\ 1 & 1 }
% --------------------------------------------------------
where $\mathcal{F}$ is a finite group scheme of order $|F|$. Proposition \ref{prop_UnconditionalBound} implies
\[
\begin{aligned}
\left| \frac{\MT(A) (\mathbb{Z}_\ell)}{\psi_\ell\left( T_{E^*}(\mathbb{Z}_\ell) \right)} \right| & = \left| \operatorname{coker} \left(\tau_\ell : \frac{T_{E^*} \otimes \mathbb{Q}_\ell}{(T')^0}(\mathbb{Z}_\ell) \to \MT(A) (\mathbb{Z}_\ell) \right) \right| \\
& \leq |\deg(\tau_\ell)|^{\dim \MT(A)} |\deg(\tau_\ell)|_\ell^{-\dim \MT(A)} \\
& = |F|^{\dim \MT(A)} |F|_\ell^{-\dim \MT(A)},
\end{aligned}
\]
which, together with equations \eqref{eq_mu} and \eqref{eq_Index}, gives the desired result.
Finally, consider part (3). As $\rho_{\ell^\infty}(J_\ell)$ is a subgroup of $\MT(A)(\mathbb{Z}_\ell)$ the index $\left[\MT(A)(\mathbb{Z}_\ell):G_{\ell^\infty} \cap \MT(A)(\mathbb{Z}_\ell)\right]$ divides $\left[ MT(A)(\mathbb{Z}_\ell) : \rho_{\ell^\infty}(J_\ell) \right]$, and we can write
\[
\begin{aligned}
\left| \frac{\MT(A)(\mathbb{Z}_\ell)}{ \rho_{\ell^\infty}(J_\ell)} \right| & \bigm\vert \mu^* \cdot \left[ \MT(A)(\mathbb{Z}_\ell) : \varphi_{\ell^\infty}(T_K(\mathbb{Z}_\ell)) \right] & \text{(by \eqref{eq_mu})}\\
& \bigm\vert \mu^* \cdot [\fieldExtension] \cdot \left| \operatorname{coker} \left(\psi_\ell : T_{E^*}(\mathbb{Z}_\ell) \to \MT(A) (\mathbb{Z}_\ell) \right) \right| & \text{(by \eqref{eq_Index})}\\
& \bigm\vert \mu^* \cdot [\fieldExtension] \cdot |F| \cdot |F|_\ell^{-2^g g!}. & \text{(by \eqref{eq_divisibility})}
\end{aligned}
\]
Since by assumption $\ell$ does not divide $|F|$ we conclude that the index $[\MT(A)(\mathbb{Z}_\ell):G_{\ell^\infty} \cap \MT(A)(\mathbb{Z}_\ell)]$ divides $\mu^* \cdot [\fieldExtension] \cdot |F|$. Finally, when $\ell$ is unramified in $K$ the factor $[\fieldExtension]$ can be omitted, cf. remark \ref{rmk_NormTotallySplit}.
\end{proof}

Starting from equations \eqref{eq_Index} and \eqref{eq_divisibility} it is also easy to prove the following result, which might have some independent interest:
\begin{proposition}\label{prop_BGK}{(Proposition \ref{prop_BGK_Intro})}
Let $A/K$ be an absolutely simple abelian variety admitting complex multiplication (over $K$) by the CM field $E$, and let $\ell$ be a prime unramified in $E$. Let $E^*$ be the reflex field of $E$ and suppose that $A$ has good reduction at all places of $K$ of characteristic $\ell$.
\begin{itemize}
\item The index $[\MT(A)(\mathbb{F}_\ell):G_\ell \cap \MT(A)(\mathbb{F}_\ell)]$ divides $[\fieldExtension] \cdot |F|$.
\item If $\ell$ is also unramified in $K$, then $[\MT(A)(\mathbb{F}_\ell):G_\ell \cap \MT(A)(\mathbb{F}_\ell)]$ divides $|F|$.
\end{itemize}
\end{proposition}

\begin{proof}
By proposition \ref{prop_GoodReductionNumberFieldTori} the hypothesis implies that $T_{E^*}$ has good reduction at $\ell$, hence the same is true for its quotient $\MT(A)$, which therefore defines a torus over $\mathbb{F}_\ell$: in particular, the group $\MT(A)(\mathbb{F}_\ell)$ makes sense and its order is not divisible by $\ell$. On the other hand, the index of $G_\ell \cap \MT(A)(\mathbb{F}_\ell)$ in $\MT(A)(\mathbb{F}_\ell)$ divides $[\fieldExtension] \cdot |F| \cdot |F|_\ell^{-2^g g!}$ by proposition \ref{prop_Epsilon} and equations \eqref{eq_Index} and \eqref{eq_divisibility}, and since $|\MT(A)(\mathbb{F}_\ell)|$ is prime to $\ell$ we deduce that $[\MT(A)(\mathbb{F}_\ell):G_\ell  \cap \MT(A)(\mathbb{F}_\ell)]$ divides $[\fieldExtension] \cdot |F|$ as claimed. The second part follows by the same argument using remark \ref{rmk_NormTotallySplit}.
\end{proof}

\section{The Mumford-Tate group in the nondegenerate case}\label{sect_NondegenerateCase}
In this section we consider CM abelian varieties $A$ with nondegenerate CM type, that is to say we assume that $\operatorname{rank}(\MT(A))=\dim A+1$: this is the ``generic'' case, and it is also known that all simple CM varieties of prime dimension have nondegenerate CM type (a result due to Ribet, cf. \cite{MR695334}). In this situation we have the following bounds on the order of $\MT(A)(\mathbb{Z}/\ell^n\mathbb{Z})$:

\begin{theorem}\label{thm_Nondegenerate}
Suppose $A$ is simple of nondegenerate CM type. For all primes $\ell \neq 2$ and all $n \geq 1$ we have
\[
(1-1/\ell)^{g+1} \cdot \ell^{(g+1)n} \leq |\MT(A)(\mathbb{Z}/\ell^n\mathbb{Z})| \leq 2^g \left(1 + 1/\ell \right)^{g-1} \ell^{(g+1)n},
\]
while for $\ell=2$ and all $n \geq 1$ we have
\[
\frac{1}{2^{2g+3}} \cdot 2^{(g+1)n} \leq |\MT(A)(\mathbb{Z}/2^n\mathbb{Z})| \leq 2^{2g-1} \cdot 2^{(g+1)n}.
\]
\end{theorem}

The proof of this result will occupy sections \ref{subsect_Nondegenerate} and \ref{subsect_Nondegenerate2}, while in sections \ref{subsect_EllipticCurves} and \ref{subsect_Surfaces} we discuss the special cases of elliptic curves and abelian surfaces.

\subsection{The natural filtration on the norm-1 torus}\label{subsect_Nondegenerate}

Let $\ell \neq 2$ be a rational prime, $L$ be a finite extension of $\mathbb{Q}_\ell$ and $\tau$ be an involution of $L$. Denote $L^\tau$ the fixed field of $\tau$, so that $L/L^\tau$ is a quadratic (Galois) extension. Fix a squarefree $d \in \mathcal{O}_{L^\tau}$ such that $L=L^\tau\left( \sqrt{d} \right)$ and consider the (multiplicative) group
\[
C=\left\{ x \in \mathcal{O}_{L}^\times \bigm\vert x \cdot \tau(x)=1 \right\}.
\]

We write $\lambda$ for a uniformizer of $L^\tau$, set $e=e \left(L^\tau/\mathbb{Q}_\ell \right)$, and consider $v_\ell$ and $v_\lambda$ as valuations on $\overline{\mathbb{Q}_\ell}$, normalized so as to have $v_\lambda(\lambda)=1$ and $v_\ell(\ell)=1$; in particular, $v_\lambda = e \cdot v_\ell$. We want to investigate the structure of the filtration of $C$ given by $C(n):=\left\{x \in C \bigm\vert v_\lambda(x-1) \geq n \right\}$. It is easy to see that every $x \in C(1)$ can be represented as
\[
x=1+2du \cdot \lambda^{2+2v} + 2 u_2 \cdot \lambda^{1+v} \sqrt{d}
\]
with $u, u_2 \in \mathcal{O}_{L^\tau}^\times$ and $v \in \mathbb{N}$ subject to the condition
\begin{equation}\label{eq_k12}
u (1 + d u \cdot \lambda^{2+2v}) = u_2^2.
\end{equation}

Furthermore, for $n \geq 1$ we have an exact sequence of abelian groups
\[
\begin{array}{ccccccc}
0 & \to & C(n+1) & \to & C(n) & \stackrel{\alpha}{\longrightarrow} & \mathcal{O}_L/(\lambda)\mathcal{O}_L,\\
  &     &        &     &  x  &\mapsto& \left[\displaystyle \frac{x-1}{2\lambda^n} \right]
\end{array}
\]
where $\left[ \cdot \right]$ denotes the class of an element of $\mathcal{O}_L$ in the quotient $\mathcal{O}_L/(\lambda)\mathcal{O}_L$. Let us describe the image of $\alpha$ for $n \geq 1$. Clearly $x \in C(n)$ implies $v \geq n-1$, and for $v \geq n$ we have $\alpha(x)=0$; when $v=n-1$ we have $\alpha(x)=[u_2\sqrt{d}]$. 
Notice now that we have an injection (of additive groups) $\frac{\mathcal{O}_{L^\tau}}{(\lambda)\mathcal{O}_{L^\tau}} \hookrightarrow \frac{\mathcal{O}_{L}}{(\lambda)\mathcal{O}_{L}}$ induced by $x  \mapsto  x\sqrt{d}$, and we claim that all points in the image of this embedding can be realized as $\alpha(x)$ for some $x \in C(n)$. This is clear for the zero element, so let us consider an element of the form $[u_2\sqrt{d}]$ with $u_2 \in \mathcal{O}_{L^\tau}^\times$.
Consider the equation 
\begin{equation}\label{eq_Findt}
t \left( 1+\lambda^{2n} d t \right) = u_2^2
\end{equation}
in the variable $t$. By Hensel's lemma, the discriminant $\Delta:=1+4u_2^2 \lambda^{2n} d$ is a square in $\mathcal{O}_{L^\tau}$ (recall that $n>0$). Let $1+z$ be the square root of $\Delta$ that is congruent to 1 modulo $\lambda$: then $z$ satisfies $(1+z)^2 = 1+4u_2^2 \lambda^{2n}d$, from which we easily find $v_\lambda(z)=2n+v_\lambda(d)$. It follows that
$
u: = \displaystyle \frac{-1+\sqrt{\Delta}}{2d \lambda^{2n}} = \frac{z}{2d \lambda^{2n}}
$
is a solution to equation \eqref{eq_Findt} which is also a $\lambda$-adic unit. We can then set $x=1+2du \cdot \lambda^{2n} + 2 u_2 \cdot \lambda^{n} \sqrt{d}$: by construction $x$ is an element of $C(n)$, and it satisfies $\alpha \left( x \right)=[u_2\sqrt{d}]$.
This shows that the image of $\alpha$ is in bijection with 
$\frac{\mathcal{O}_{L^\tau}}{(\lambda)\mathcal{O}_{L^\tau}}$. 
Finally, a very similar argument can be repeated when $\ell=2$, except that Hensel's lemma is now only applicable for $n>v_\lambda(2)$. We thus deduce the following lemma:

\begin{lemma}\label{lemma_Filtration}
Suppose $\ell \neq 2$. For every $n \geq 1$, the quotient $C(n)/C(n+1)$ has order $\left| \frac{\mathcal{O}_{L^\tau}}{(\lambda)\mathcal{O}_{L^\tau}} \right|$. For $\ell=2$ the same conclusion holds for every $n>v_\lambda(2)$.
\end{lemma}

The quotients $C(n)/C(n+1)$ for small values of $n$ are described by the following lemma:
\begin{lemma}\label{lemma_Level0}
Let $f$ be the inertia degree of $L^\tau$ over $\mathbb{Q}_\ell$. Suppose first $\ell \neq 2$: then the quotient $\frac{C(0)}{C(1)}$ has order either $2\ell^{f}$ or $\ell^{f}+1$, with the first (resp. second) case happening exactly when $L/L^\tau$ is ramified (resp. unramified). Suppose on the other hand that $\ell=2$ and $n \leq v_\lambda(2)$: then the quotient $\frac{C(n)}{C(n+1)}$ has order at most $4^{f}$.
\end{lemma}
Before giving a proof, recall the following
\begin{definition}
Let $L$ be a finite extension of $\mathbb{Q}_\ell$ with ring of integers $\mathcal{O}_L$ and residue field $\mathbb{F}$. Let $\pi:\mathcal{O}_L \to \mathbb{F}$ be the canonical projection. The \textbf{Teichm\"uller lift} is the unique group homomorphism $\omega:\mathbb{F}^\times \to \mathcal{O}_L^\times$ such that, for all $y \in \mathbb{F}^\times$, the element $\omega(y) \in \mathcal{O}_L^\times$ is the unique solution to the equation $x^{|\mathbb{F}|-1}=1$ satisfying $\pi(x)=y$.
\end{definition}

\begin{proof}
Consider first the case of $L/L^\tau$ being unramified (and $\ell \neq 2$). 
Let $\pi : \mathcal{O}_{L} \to \mathbb{F}:=\frac{\mathcal{O}_L}{(\lambda)\mathcal{O}_L}$ be the canonical projection, and observe that $\mathbb{F}$ has order $\ell^{2f}$. It is clear that $\pi$ restricts to a map $C(0) \to \mathbb{F}^\times$, and on the other hand $x \in C(0)$ maps to 1 if and only if $v_{\lambda}(x-1)>0$, i.e. if and only if $x \in C(1)$: this implies that $C(0)/C(1)$ injects into $\mathbb{F}^\times$. The involution $\tau$ induces on $\mathbb{F}$ an automorphism $\tau_{\mathbb{F}}$, which is necessarily the unique nontrivial involution $x \mapsto x^{\ell^{f}}$.
Let now $x \in C(0)$. By definition we have $x \cdot \tau(x)=1$, hence
\[
1=\pi(x) \cdot \pi(\tau(x))=\pi(x) \cdot \tau_{\mathbb{F}}(\pi(x))=\pi(x)^{\ell^{f}+1},
\] so $C(0)/C(1)$ injects into the subgroup $H$ of $\mathbb{F}^\times$ consisting of the roots of unity of order dividing $\ell^{f}+1$. The group $H$ is of order $\ell^{f}+1$, and it is not hard to see that $C(0)/C(1)$ surjects onto it: indeed for every $h \in H$ we have $\omega(h) \in C(0)$, and by definition $\pi(\omega(h))=h$. 
Suppose on the other hand that $L/L^\tau$ is ramified, so that $L=L^\tau(\sqrt{d})$ with $v_\lambda(d)=1$. Again we see that $C(0)/C(1)$ injects into $\mathbb{F}^\times:=\left(\frac{\mathcal{O}_L}{(\lambda)\mathcal{O}_L}\right)^\times$ (which however is not a field anymore), and the involution $\tau$ acts on an element $[a+b\sqrt{d}] \in \left(\frac{\mathcal{O}_L}{(\lambda)\mathcal{O}_L}\right)^\times$, with $a,b \in \mathcal{O}_{L^\tau}$, by sending it to $[a-b\sqrt{d}]$. Writing $\pi(x)=[a+b\sqrt{d}]$, the equation $x\tau(x)=1$ implies $[a^2-db^2]=1$, which in turn, since $v_\lambda(d)=1$, means $[a^2]=1$ and $[a] = \pm 1$. This shows that $C(0)/C(1)$ injects into $\{\pm 1\} \times \frac{\mathcal{O}_{L^\tau}}{(\lambda)\mathcal{O}_{L^\tau}}$, a set of order $2 \cdot \ell^f$. On the other hand, for any value of $[\pm 1 +b\sqrt{d}] \in \mathbb{F}^\times$, the equation $a^2=1+db^2$ (with fixed $b$, in the variable $a$) admits solutions in $\mathcal{O}_{L^\tau}$ by Hensel's lemma;
the elements $\pm a+b\sqrt{d} \in C(0)$ then satisfy
\[
(\pm a+b\sqrt{d}) \cdot \tau(\pm a+b\sqrt{d})=a^2-db^2 = 1,
\]
and on the other hand $\pi(\pm a+b\sqrt{d}) = [\pm 1 + b\sqrt{d}]$, so $C(0)/C(1)$ actually projects surjectively on $\{\pm 1\} \times \frac{\mathcal{O}_{L^\tau}}{(\lambda)\mathcal{O}_{L^\tau}}$; this shows that $|C(0)/C(1)|=2\ell^f$ as claimed.
The upper bound for $\ell=2$ likewise follows from the fact that for any $n \geq 0$ the quotient $C(n)/C(n+1)$ injects into $\frac{\mathcal{O}_L}{(\lambda)\mathcal{O}_L}$.
\end{proof}

\subsection{The order of $\MT(A)(\mathbb{Z}/\ell^n\mathbb{Z})$}\label{subsect_Nondegenerate2}
Let $E$ be a CM field of degree $2g$ over $\mathbb{Q}$ and $T_E$ be the associated algebraic torus, and let $\tau$ denote complex conjugation on $E$. If $A/K$ is an abelian variety with complex multiplication by the nondegenerate CM type $(E,S)$, it is known that we have
\[
\MT(A)(B) = \left\{ x \in (E \otimes_\mathbb{Q} B)^\times \bigm\vert x \tau(x) \in B^\times \right\} \quad \forall \; \mathbb{Q}\text{-algebra } B.
\]

We can also consider the `norm 1' (or Hodge) subtorus of $\MT(A)$ given as a functor by
\[
\Hg(A)(B) = \left\{ x \in (E \otimes_\mathbb{Q} B)^\times \bigm\vert x \tau(x) =1 \right\} \quad \forall \; \mathbb{Q}\text{-algebra } B.
\]

We aim to give bounds on the number of $\frac{\mathbb{Z}}{\ell^n\mathbb{Z}}$-points of $\MT(A)$, but it is easier to first consider $\Hg(A)$. If we write $E \otimes \mathbb{Q}_\ell \cong \prod_{i=1}^s F_i$ (a product of fields), we have
\[
\Hg(A)(\mathbb{Q}_\ell) = \left\{ x=(x_1,\ldots,x_s) \in \prod_{i=1}^s F_i^\times \bigm\vert x \tau(x) =1 \right\}.
\]

We can renumber the $F_i$'s in such a way that $\tau$ acts by exchanging $F_{2i-1}$ and $F_{2i}$ for $i=1,\ldots,r$ and it acts as an involution on $F_i$ for $i=2r+1,\ldots,s$. With this convention, a point $(x_1,\ldots,x_{2r},x_{2r+1}, \ldots,x_s) \in \prod_{i=1}^s F_i^\times$ is in $\Hg(A)(\mathbb{Q}_\ell)$ if and only if $x_{2i-1}x_{2i}=1$ for $i=1,\ldots,r$ and $x_i\tau(x_i)=1$ for $i=2r+1,\ldots,s$, that is, 
\begin{equation}\label{eq_Hg}
\Hg(A)(\mathbb{Q}_\ell) \cong \prod_{i=1}^r \left\{ x_{2i-1} \in F_{2i-1}^\times \right\} \times \prod_{i=2r+1}^s \left\{ x_i \in F_i^\times \bigm\vert x_i \tau(x_i) =1 \right\}.
\end{equation}

The character groups of $\MT(A)_{\mathbb{Q}_\ell}$ and of $\Hg(A)_{\mathbb{Q}_\ell}$ are quotients of $\widehat{T_{E,\mathbb{Q}_\ell}}$, which in turn is generated by elements of the form $(\chi_1,\ldots,\chi_s)$, where $\chi_i$ ranges over the embeddings of $F_i$ in $\overline{\mathbb{Q}_\ell}$. It follows that a point $x \in \Hg(A)(\mathbb{Q}_\ell)$ is in $\Hg(A)(\mathbb{Z}_\ell)$ if and only if for any choice of embeddings $\chi_i:F_i \hookrightarrow \overline{\mathbb{Q}_\ell}$ we have $\chi_i(x_i) \in \mathcal{O}_{\chi_i(F_i)}$; as the property of being $\ell$-integral is Galois-invariant we deduce that a necessary and sufficient condition is $x_i \in \mathcal{O}_{F_i}^\times$. Hence we find $\displaystyle \Hg(A)(\mathbb{Z}_\ell) \cong \prod_{i=1}^r \mathcal{O}_{F_{2i-1}}^\times \times \prod_{i=2r+1}^s \left\{ x_i \in \mathcal{O}_{F_i}^\times \bigm\vert x_i \tau(x_i) =1 \right\}$, and a perfectly analogous argument shows that
\[
\begin{aligned}
\Hg(A)(1+\ell^n\mathbb{Z}_\ell) \cong \prod_{i=1}^r & \left\{ x_{2i-1} \in \mathcal{O}_{F_{2i-1}}^\times \bigm\vert v_\ell(x_{2i-1}-1) \geq n \right\} \times \\ & \times \prod_{i=2r+1}^s \left\{ x_i \in \mathcal{O}_{F_i}^\times \bigm\vert v_\ell(x_{i}-1) \geq n, \; x_i \tau(x_i) =1 \right\}.
\end{aligned}
\]

Write $e_i$ and $f_i$ for the ramification index and inertia degree of $F_i^\tau$ over $\mathbb{Q}_\ell$, and $\lambda_i$ for a uniformizer of $F_i^\tau$ ($i=2r+1,\ldots,s$). The order of $\left|\frac{\Hg(A)(\mathbb{Z}_\ell)}{\Hg(A)(1+\ell^n\mathbb{Z}_\ell)}\right|$ is then given by
\begin{equation}\label{eq_OrderHodge}
\displaystyle \left|\Hg(A)(\mathbb{Z}/\ell^n\mathbb{Z})\right| = \prod_{i=1}^r \left|\frac{\mathcal{O}_{F_{2i-1}}^\times}{1+\ell^n \mathcal{O}_{F_{2i-1}}} \right| \; \times \; \prod_{i=2r+1}^s \left|\frac{C^{(i)}(0)}{C^{(i)}(ne_i)}\right|,\end{equation}
where
\[
C^{(i)}(k)=\left\{ x_i \in \mathcal{O}_{F_i}^\times \bigm\vert v_{\lambda_i}(x_{i}-1) \geq k, \; x_i \cdot \tau(x_i) =1 \right\}
\]
is the filtration we studied in the previous section for the field $F_i$ and the involution $\tau|_{F_i}$. For $i=1,\ldots,r$ let furthermore $\pi_{i}$ (resp. $e_i, f_i$) be a uniformizer (resp. the ramification index over $\mathbb{Q}_\ell$, the inertia degree over $\mathbb{Q}_\ell$) of $F_{2i-1}$. We now compute the order of $\Hg(A)(\mathbb{Z}/\ell^n \mathbb{Z})$. Basic properties of local fields show that for $i=1,\ldots,r$ the quotient $\left|\frac{\mathcal{O}_{F_{2i-1}}^\times}{1+\ell^n\mathcal{O}_{F_{2i-1}}}\right|$ has order
\[
\left|\frac{\mathcal{O}_{F_{2i-1}}^\times}{1+\pi_{i}\mathcal{O}_{F_{2i-1}}}\right| \cdot \prod_{j=1}^{ne_i-1} \left| \frac{1+(\pi_{i})^{j} \mathcal{O}_{F_{2i-1}}}{1+(\pi_{i})^{j+1} \mathcal{O}_{F_{2i-1}}} \right| = \left(\ell^{f_i}-1\right) \cdot \ell^{f_i(ne_i-1)},
\]
while (for $\ell \neq 2$) lemma \ref{lemma_Filtration} gives
\[
\left|\frac{C^{(i)}(0)}{C^{(i)}(ne_i)}\right| =  \left|\frac{C^{(i)}(0)}{C^{(i)}(1)}\right| \cdot \left|\frac{C^{(i)}(1)}{C^{(i)}(ne_i)}\right| =  \left|\frac{C^{(i)}(0)}{C^{(i)}(1)}\right| \cdot \ell^{f_i(ne_i-1)}.
\]
Now notice that $s-2r$ does not exceed $g$: indeed $[F_i:F_i^\tau]=2$ for every $i=2r+1, \ldots, r$, hence $\displaystyle
2g = \left[ E \otimes \mathbb{Q}_\ell : \mathbb{Q}_\ell \right] \geq \sum_{i=2r+1}^s \left[ F_i:\mathbb{Q}_\ell \right] \geq 2(s-2r)$
as claimed. Applying lemma \ref{lemma_Level0} we then deduce that the order of $\Hg(A)(\mathbb{Z}/\ell^n \mathbb{Z})$ is at most
\[
\begin{aligned}
\prod_{i=1}^r \ell^{nf_ie_i} \cdot & \prod_{i=2r+1}^s  2\left(1+\ell^{f_i}\right) \left(\ell^{f_i}\right)^{ne_i-1}  \\ & =2^{s-2r} \prod_{i=1}^{2r} \ell^{\frac{1}{2}n [F_{i}:\mathbb{Q}_\ell]}  \prod_{i=2r+1}^s \left(1+\ell^{-f_i}\right) \left(\ell^{f_i}\right)^{ne_i} \\
& \leq 2^{s-2r} \left(1+1/\ell\right)^{s-2r} \prod_{i=1}^{s} \ell^{\frac{1}{2}n [F_{i}:\mathbb{Q}_\ell]} \\ & \leq 2^{g} \left(1+1/\ell\right)^{g} \ell^{gn},
\end{aligned}
\]
and at least
\[
\begin{aligned}
\prod_{i=1}^r \left(\ell^{f_i}-1\right) \ell^{(ne_i-1)f_i} & \prod_{i=2r+1}^s (\ell^{f_i}+1)\left(\ell^{f_i}\right)^{ne_i-1} \\ & \geq (1-1/\ell)^r \cdot \prod_{i=1}^r \ell^{nf_ie_i} \prod_{i=2r+1}^s \ell^{nf_ie_i}
\\ & = (1-1/\ell)^r \cdot \prod_{i=1}^{2r} \ell^{\frac{1}{2}n [F_{i}:\mathbb{Q}_\ell]} \times  \prod_{i=2r+1}^{s} \ell^{\frac{1}{2}n [F_{i}:\mathbb{Q}_\ell]} \\
& \geq (1-1/\ell)^g \cdot \ell^{gn};
\end{aligned}
\]
moreover, if for at least one index $i\in\{2r+1,\ldots,s\}$ the extension $F_i/F_i^\tau$ is ramified, then we see from lemma \ref{lemma_Level0} that the lower bound can be improved to
\begin{equation}\label{eq_ImprovedLowerBound}
\left|\Hg(A)(\mathbb{Z}/\ell^n \mathbb{Z})\right| \geq 2(1-1/\ell)^g \cdot \ell^{gn}.
\end{equation}
To finish the proof of theorem \ref{thm_Nondegenerate} we shall use the following result:
\begin{lemma}\label{lemma_ImPsi}
Consider the map
\[
\begin{array}{cccc}
\Psi: & \Hg(A)(\mathbb{Z}/\ell^n\mathbb{Z}) \times \left( \mathbb{Z}/\ell^n \mathbb{Z} \right)^\times & \to & \MT(A)(\mathbb{Z}/\ell^n\mathbb{Z})
\\
& (h, m) & \mapsto & m^{-1}h.
\end{array}
\]
If $\ell \neq 2$, the group $\operatorname{Im} \Psi$ has order equal to $\frac{1}{2} \left|\Hg(A)(\mathbb{Z}/\ell^n\mathbb{Z}) \right| \times (1-1/\ell)\ell^{n}$ and has index at most 2 in $\MT(A)(\mathbb{Z}/\ell^n\mathbb{Z})$. Moreover, $\Psi$ is surjective if and only if for all $x \in \MT(A)(\mathbb{Z}/\ell^n\mathbb{Z})$ the number $x \tau(x)$ is a square in $\left(\mathbb{Z}/\ell^n\mathbb{Z} \right)^\times$. On the other hand, for $\ell=2$ we have
\begin{itemize}
\item for $n=1$, the group $\operatorname{Im} \Psi$ has order equal to that of $\left|\Hg(A)(\mathbb{Z}/2\mathbb{Z}) \right|$ and $\Psi$ is surjective;
\item for $n=2$, the group $\operatorname{Im} \Psi$ has order equal to that of $\left|\Hg(A)(\mathbb{Z}/4\mathbb{Z}) \right|$ and $\operatorname{Im} \Psi$ has index either 1 or $2$ in $\MT(A)(\mathbb{Z}/4\mathbb{Z})$;
\item for $n \geq 3$, the group $\operatorname{Im} \Psi$ has order equal to $2^{n-3} \cdot \left|\Hg(A)(\mathbb{Z}/2^n\mathbb{Z}) \right|$ and $\operatorname{Im} \Psi$ has index $1$, $2$ or $4$ in $\MT(A)(\mathbb{Z}/2^n\mathbb{Z})$;
\end{itemize}
\end{lemma}
\begin{proof}
Let us start with the case $\ell \neq 2$. The kernel of $\psi$ is given by the intersection of $\Hg(A)(\mathbb{Z}/\ell^n\mathbb{Z})$ and $\left( \mathbb{Z}/\ell^n \mathbb{Z} \right)^\times$ inside $\MT(A)(\mathbb{Z}/\ell^n\mathbb{Z})$, namely 
\[
\left\{ h \in \left(\mathbb{Z}/\ell^n \mathbb{Z}\right)^\times \bigm\vert h \tau(h)=1 \right\} = \left\{ h \in \left(\mathbb{Z}/\ell^n \mathbb{Z}\right)^\times \bigm\vert h ^2=1 \right\}=\left\{ \pm 1 \right\},
\]
so $\operatorname{Im} \Psi$ has order
\[
\frac{1}{\left| \operatorname{ker} \Psi \right|} \cdot |\Hg(A)(\mathbb{Z}/\ell^n\mathbb{Z})| \cdot \left| \left( \mathbb{Z}/\ell^n\mathbb{Z} \right)^\times \right| = \frac{(\ell-1)\ell^{n-1}}{2} \cdot |\Hg(A)(\mathbb{Z}/\ell^n\mathbb{Z})|
\]
as claimed.

As for the index of $\operatorname{Im} \Psi$, notice first that for every $x=m^{-1}h \in \operatorname{Im} \Psi$ we have that $x\cdot \tau(x)=m^{-2}$ is a square in $\left(\mathbb{Z}/\ell^n\mathbb{Z}\right)^\times$, so if $\Psi$ is surjective we necessarily have $x \cdot \tau(x) \in (\mathbb{Z}/\ell^n\mathbb{Z})^{\times 2}$ for every $x \in \MT(A)(\mathbb{Z}/\ell^n\mathbb{Z})$. Conversely, suppose that for every $x$ in $\MT(A)(\mathbb{Z}/\ell^n\mathbb{Z})$ the number $x \tau(x)$ is a square in $\left(\mathbb{Z}/\ell^n \mathbb{Z}\right)^\times$, say $x \tau(x)=\mu(x)^2$ with $\mu(x) \in \left(\mathbb{Z}/\ell^n\mathbb{Z} \right)^\times$. Then every $x$ can be written as $x = \mu(x) \cdot \frac{x}{\mu(x)}$, and since $\frac{x}{\mu(x)}$ is in $\Hg(A)(\mathbb{Z}/\ell^n \mathbb{Z})$ this shows that $x$ belongs to $\operatorname{Im} \Psi$, which is therefore surjective.

Finally, if there is a $y \in \MT(A)(\mathbb{Z}/\ell^n\mathbb{Z})$ such that $y \tau(y)$ is not a square in $\left(\mathbb{Z}/\ell^n\mathbb{Z}\right)^\times$, then using the fact that $\left(\mathbb{Z}/\ell^n\mathbb{Z}\right)^{\times 2}$ is of index 2 in $\left(\mathbb{Z}/\ell^n\mathbb{Z}\right)^{\times}$ we easily see that for every $x \in \MT(A)(\mathbb{Z}/\ell^n\mathbb{Z})$ either $x$ or $xy$ belongs to $\operatorname{Im} \Psi$, thus proving the remaining claim.
The conclusion for $\ell=2$ follows by the same argument upon noticing that $\frac{\left(\mathbb{Z}/2^n\mathbb{Z} \right)^\times}{\left(\mathbb{Z}/2^n\mathbb{Z} \right)^{\times 2}}$ has order $1, 2$, or $4$, according to whether $n$ is 1, 2, or at least 3.
\end{proof}

Combining this last lemma with our previous estimates gives the desired upper bound
\[
\begin{aligned}
\left|\MT(A)(\mathbb{Z}/\ell^n\mathbb{Z}) \right| \leq 2 \left| \operatorname{Im} \Psi \right| & = \left|\Hg(A)(\mathbb{Z}/\ell^n\mathbb{Z})\right| \times \left| \left( \mathbb{Z}/\ell^n \mathbb{Z} \right)^\times \right| \\ & \leq 2^{g} \left(1+1/\ell \right)^{g-1}\ell^{(g+1)n}.
\end{aligned}
\]
As for the lower bound, suppose first that for at least one index $i$ in the set $\{2r+1,\ldots,s \}$ the extension $L_i/L_i^\tau$ is ramified: then using the lower bound of equation \eqref{eq_ImprovedLowerBound} (which is conditional on this hypothesis) we find
\[
\left|\MT(A)(\mathbb{Z}/\ell^n\mathbb{Z}) \right| \geq \frac{1}{2} \left|\Hg(A)(\mathbb{Z}/\ell^n\mathbb{Z})\right| \times \left| \left( \mathbb{Z}/\ell^n \mathbb{Z} \right)^\times \right| \geq (1-1/\ell)^{g+1}\ell^{(g+1)n}.
\]
Suppose on the other hand that $L_i/L_i^\tau$ is unramified for every $i=2r+1,\ldots,s$: then we claim that map $\Psi$ from lemma \ref{lemma_ImPsi} is not surjective. Assuming this is the case, we have
\[
\left|\MT(A)(\mathbb{Z}/\ell^n\mathbb{Z}) \right| \geq 2 \times \frac12 \times \left|\Hg(A)(\mathbb{Z}/\ell^n\mathbb{Z})\right| \times \left| \left( \mathbb{Z}/\ell^n \mathbb{Z} \right)^\times \right| \geq (1-1/\ell)^{g+1}\ell^{(g+1)n},
\]
which is what we want to show. We are thus reduced to proving that $\Psi$ is not surjective, or equivalently (by lemma \ref{lemma_ImPsi}), to showing that there is an $x \in \MT(A)(\mathbb{Z}/\ell^n\mathbb{Z})$ such that $x \tau(x)$ is not a square in $\left(\mathbb{Z}/\ell^n\mathbb{Z}\right)^\times$. By the same argument that leads to equations \eqref{eq_Hg} and \eqref{eq_OrderHodge}, we can represent elements of $\MT(A)(\mathbb{Z}_\ell)$ as tuples
\[
(x_1,\ldots,x_{2r}, x_{2r+1},\ldots,x_{s},m) \in \prod_{i=1}^{2r} \mathcal{O}_{F_{i}}^\times \times \prod_{j=2r+1}^s \mathcal{O}_{F_j}^\times \times \mathbb{Z}_\ell^\times,
\]
satisfying $
x_{2i-1}x_{2i}=m$ for $i=1,\ldots,r$ and $x_j\tau(x_j) =m$ for $j=2r+1,\ldots,s$. Now if $2r=s$ it is clear that $\operatorname{MT}(A)(\mathbb{Z}/\ell^n\mathbb{Z})$ contains elements $x$ such that $x \tau(x)$ is not a square (it suffices to choose $m \in \mathbb{Z}_\ell^\times$ which is not a square in $(\mathbb{Z}/\ell^n\mathbb{Z})^\times$ and set $x_{2i-1}=1, x_{2i}=m$ for $i=1,\ldots,r$), so we can assume $s>2r$.
For $j=2r+1,\ldots,s$ write $F_j=F_j^\tau(\sqrt{d_j})$ for some squarefree $d_j \in \mathcal{O}_{F_j}^\times$ (recall that we assume $F_j/F_j^\tau$ to be unramified), and likewise write $x_j=a_j+b_j\sqrt{d_j}$ for some $a_j, b_j \in \mathcal{O}_{F_j^\tau}$. We claim that since $F_j/F_j^\tau$ is unramified every element $m \in \mathbb{Z}_\ell^\times$ can be represented as $a_j^2-d_jb_j^2$ for some choice of $a_j, b_j \in \mathcal{O}_{F_j^\tau}$. To see this, notice that for fixed $m$ and $d_j$ the conic section $\mathcal{C} : \{a^2-d_jb^2=mc^2\}$ admits a point $(a_0,b_0,c_0)$ over the residue field of $F_j^\tau$; as $d_j$ is not a square in $F_j^\tau$ we cannot have $c_0=0$, and since $\mathcal{C}$ is smooth the point $(a_0,b_0,c_0)$ lifts to a point $(a,b,c) \in \mathcal{C}\left(\mathcal{O}_{F_j^\tau}\right)$, with $c$ a unit (since it does not reduce to 0 in the residue field). Dividing through by $c^2$ then yields
$
(a/c)^2-d_j(b/c)^2=m
$
as desired.
Pick now a fixed non-square $m \in \mathbb{Z}_\ell^\times$ and for each $j=2r+1,\ldots,s$ fix a representation $m=a_j^2-d_jb_j^2$. Take furthermore $x_{2i-1}=1, x_{2i}=m$ for $i=1,\ldots,r$. 

The corresponding point $x=((x_{i})_{i=1,\ldots,2r},(x_j)_{j=2r+1,\ldots,s},m)$ of $\MT(A)(\mathbb{Z}_\ell)$ has the property that $x \tau(x)=m$ is not a square in $\mathbb{Z}_\ell^\times$, and therefore  the image of $x$ in $\MT(A)(\mathbb{Z}/\ell^n\mathbb{Z})$ has again the property that $x \tau(x)=[m] \in \left(\mathbb{Z}/\ell^n\mathbb{Z}\right)^\times$ is not a square. Combined with lemma \ref{lemma_ImPsi}, this shows that $\Psi$ is not surjective in this case and concludes the proof of theorem \ref{thm_Nondegenerate} for $\ell \neq 2$.

Notice now that for $\ell=2$ the lower bound of theorem \ref{thm_Nondegenerate} is trivial for $n \leq 2$, so we can assume $n \geq 3$. We then remark that (by equation \eqref{eq_OrderHodge}) $\Hg(A)(\mathbb{Z}/2^n\mathbb{Z})$ has order at least
\[
\prod_{i=1}^r 2^{(n-1)[F_{2i-1}:\mathbb{Q}_\ell]} \; \times \; \prod_{i=2r+1}^s \left| \frac{C^{(i)}(e_i+1)}{C^{(i)}(ne_i)} \right|,
\]
which (by the same argument as above, using the second part of lemma \ref{lemma_Filtration}) in turn is at least
\[
\prod_{i=1}^r 2^{(n-1)[F_{2i-1}:\mathbb{Q}_\ell]} \; \times \; \prod_{i=2r+1}^s \left( 2^{f_i} \right)^{(n-1)e_i-1} \geq 2^{g(n-2)}.
\]
Furthermore, taking into account the factor coming from the homotheties -- namely $\left(\mathbb{Z}/2^n\mathbb{Z} \right)^\times$ -- we find $|\MT(A)(\mathbb{Z}/2^n\mathbb{Z})| \geq 2^{(g+1)(n-2)-1}$. Finally, the upper bound for $\ell=2$ follows trivially from the previous computations and from the second halves of lemmas \ref{lemma_Level0} and \ref{lemma_ImPsi}.

\subsection{Elliptic curves}\label{subsect_EllipticCurves}
When the CM abelian variety under consideration is an elliptic curve we can give a complete description of the full adelic Galois representation:

\begin{theorem}\label{thm_EC}
Let $A/K$ be an elliptic curve such that $\operatorname{End}_{\overline K}(A)$ is an order in an imaginary quadratic field $E$. Denote by $\displaystyle 
\rho_\infty : \abGal{K} \to \prod_{\ell} \operatorname{Aut} T_\ell A$ 
the natural adelic representation attached to $A$, and let $G_\infty$ be its image. For every prime $\ell$ denote by $C_\ell$ the group $\left(\mathcal{O}_E \otimes \mathbb{Z}_\ell \right)^\times$, considered as a subgroup of $\operatorname{Aut}_{\mathbb{Z}_\ell} \left( \mathcal{O}_E \otimes \mathbb{Z}_\ell \right) \cong \operatorname{GL}_2(\mathbb{Z}_\ell) \cong \operatorname{Aut} T_\ell A$, 
and let $N(C_\ell)$ be the normalizer of $C_\ell$ in $\operatorname{GL}_2(\mathbb{Z}_\ell)$.
\begin{enumerate}
\item Suppose that $E \subseteq K$: then $G_\infty$ is contained in $\prod_\ell C_\ell$, and the index $\left[\prod_\ell C_\ell : G_\infty \right]$ does not exceed $3[K:\mathbb{Q}]$. The equality $G_{\ell^\infty}=C_\ell$ holds for every prime $\ell$ unramified in $K$ and such that $A$ has good reduction at all places of $K$ of characteristic $\ell$.
\item Suppose that $E \not \subseteq K$: then $G_\infty$ is contained in $\prod_\ell N(C_\ell)$ but not in $\prod_\ell C_\ell$, and the index $\left[\prod_\ell N(C_\ell) : G_\infty \right]$ is not finite. The intersection $H_\infty=G_\infty \cap \prod_\ell C_\ell$ has index 2 in $G_\infty$, and the index $\left[\prod_\ell C_\ell : H_\infty \right]$ does not exceed $6[K:\mathbb{Q}]$. The equality $G_{\ell^\infty}=N(C_\ell)$ holds for every prime $\ell$ unramified in $K \cdot E$ and such that $A$ has good reduction at all places of $K$ of characteristic $\ell$.
\end{enumerate}
Finally, the constants 3 and 6 appearing in parts (1) and (2) respectively can be replaced by 1 and 2 if we further assume that the $j$-invariant of $A$ is neither 0 nor 1728.
\end{theorem}

We start by recording the following consequence of theorem \ref{thm_Finale}:

\begin{corollary}\label{cor_EC}
Let $A/K$ be an elliptic curve admitting complex multiplication (over $K$) by the imaginary quadratic field $E$. The group $G_{\ell^\infty}$ is contained in $\MT(A)(\mathbb{Z}_\ell)=C_\ell$, and if $A$ has good reduction at all places of $K$ of characteristic $\ell$ the index $[\MT(A)(\mathbb{Z}_\ell):G_{\ell^\infty}]$ is at most $\frac{1}{2}[K:\mathbb{Q}]$. If in addition $\ell$ is also unramified in $K$ we have $G_{\ell^\infty}=\left(\mathcal{O}_E \otimes \mathbb{Z}_\ell \right)^\times$.
\end{corollary}
\begin{proof}
Since $E$ is quadratic, $E$ and $E^*$ coincide and the reflex norm is simply the identity $T_E \to T_E$, hence $\MT(A)=T_E$ and (in the notation of theorem \ref{thm_Finale}) $F$ is the trivial group. In particular $T_E(\mathbb{Z}_\ell)=\left(\mathcal{O}_E \otimes \mathbb{Z}_\ell \right)^\times=C_\ell$ contains $G_{\ell^\infty}$ by \cite[Corollary 2 to Theorem 5]{MR0236190} (cf. also \cite[Corollaire on p. 302]{MR0387283}): the claim on the index then follows from theorem \ref{thm_Finale} upon noticing that $[K:E^*]=[K:E]=\frac{1}{2}[K:\mathbb{Q}]$.
Furthermore, if $\ell$ is unramified in $K$, then it is also unramified in $E$, and the remaining assertion $G_{\ell^\infty}=C_\ell=\MT(A)(\mathbb{Z}_\ell)$ follows from part (3) of theorem \ref{thm_Finale}.
\end{proof}

We shall also need some results concerning elliptic curves $A/K$ that admit complex multiplication over $\overline{K}$ but not over $K$. We start with the following easy properties of $N(C_\ell)$:

\begin{lemma}\label{lemma_NormalizerCartan}
$C_\ell$ is of index 2 in $N(C_\ell)$. In particular, $N(C_\ell)$ is generated by $C_\ell$ and any element in $N(C_\ell) \setminus C_\ell$. Furthermore, if $H_\ell$ is an open subgroup of $C_\ell$, then the normalizer of $H_\ell$ in $\operatorname{GL}_2(\mathbb{Z}_\ell)$ is contained in $N(C_\ell)$.
\end{lemma}
\begin{proof}
Fix $\omega \in \mathcal{O}_E$ such that $(1,\omega)$ is a $\mathbb{Z}$-basis of $\mathcal{O}_E$. There exist $c,d \in \mathbb{Z}$ such that $\omega$ satisfies the quadratic relation $\omega^2=c\omega+d$. In the $\mathbb{Z}_\ell$-basis $(1,\omega)$ of $\mathcal{O}_E \otimes \mathbb{Z}_\ell$, the group $C_\ell$ is the subgroup of $\operatorname{GL}_2(\mathbb{Z}_\ell)$ given by the invertible matrices that can be written as $\left(\begin{matrix} a & bd \\ b & a+bc \end{matrix} \right)$ for some $a,b \in \mathbb{Z}_\ell$. We thus see that $C_\ell$ is given by the intersection of $\operatorname{GL}_2(\mathbb{Z}_\ell)$ with a 2-dimensional plane $\Pi$ (that defined by the equations $x_{11}+cx_{21}=x_{22}, x_{12}=dx_{21}$, where $x_{ij}$ is the coefficient on the $i$-th row and $j$-th column). In particular, for an element $g \in \operatorname{GL}_2(\mathbb{Z}_\ell)$ the condition of normalizing $C_\ell$ is equivalent to that of stabilizing $\Pi$. The latter is a Zariski-closed condition, and since any subgroup $H_\ell$ of $C_\ell$ open in the $\ell$-adic topology is Zariski-dense in $\Pi$ we see that if $g$ normalizes $H_\ell$, then it stabilizes $\Pi$ and hence it normalizes $C_\ell$.
Finally, with the explicit description at hand it is immediate to see that $[N(C_\ell):C_\ell]=2$, and that a nontrivial element of $N(C_\ell) \setminus C_\ell$ is given by $\left(\begin{matrix} 1 & c \\ 0 & -1 \end{matrix} \right)$.
\end{proof}

\begin{lemma}\label{lemma_ContainmentEC}
Suppose $A/K$ is an elliptic curve such that $\operatorname{End}_K(A)=\mathbb{Z}$ but $\operatorname{End}_{\overline{K}}(A)$ is an order in an imaginary quadratic field $E$: then for every prime $\ell$ the group $G_{\ell^\infty}$ is contained in $N(C_\ell)$.
\end{lemma}
\begin{proof}
The field $K^1=K \cdot E$ is a quadratic extension of $K$ over which all the endomorphisms of $A$ are defined, and the group $G_{\ell^\infty}^1 = \rho_{\ell^\infty}\left(\abGal{K^1} \right)$ is a closed subgroup of $G_{\ell^\infty}$ of index at most 2 (hence in particular it is normal and open in $G_{\ell^\infty}$). Let $R=\operatorname{End}_{\overline{K}}(A)$. Since $A$ admits complex multiplication by $R$ over $K^1$, we know by \cite[§4.5, Corollaire]{MR0387283} that $G_{\ell^\infty}^1$ is of finite index in $(R \otimes \mathbb{Z}_\ell)^\times$, which in turn is of finite index in $C_\ell$. Thus the normalizer of $G_{\ell^\infty}^1$ is included in $N(C_\ell)$ by lemma \ref{lemma_NormalizerCartan}, and since $G_{\ell^\infty}^1$ is normal in $G_{\ell^\infty}$ we have $G_{\ell^\infty} \subseteq N(G_{\ell^\infty}^1) \subseteq N(C_\ell)$ as claimed.
\end{proof}

\begin{lemma}\label{lemma_NontrivialIntersection}
In the situation of the previous lemma, for all primes $\ell$ the group $G_{\ell^\infty}$ has nontrivial intersection with $N(C_\ell) \setminus C_\ell$.
\end{lemma}
\begin{proof}
For all primes $\ell$ we have $G_{\ell^\infty} \subseteq N(C_\ell)$. On the other hand, we know by Faltings' theorem that the centralizer of $G_{\ell^\infty}$ in $\operatorname{End} \left(T_\ell A \right) \otimes \mathbb{Q}_\ell$ equals $\operatorname{End}_K(A) \otimes \mathbb{Q}_\ell=\mathbb{Q}_\ell$. It follows that $G_{\ell^\infty}$ cannot be abelian, for otherwise its centralizer would contain all of $G_{\ell^\infty}$ (which is not contained in the homotheties $\mathbb{Q}_\ell$): in particular, $G_{\ell^\infty}$ must have nontrivial intersection with $N(C_\ell) \setminus C_\ell$. 
\end{proof}

We can now prove theorem \ref{thm_EC}:

\begin{proof}{(of theorem \ref{thm_EC})}
The proof is quite similar to that of theorem \ref{thm_Finale}, the main differences being that we need to treat all places at the same time and that the action of $E$ needs not be defined over $K$.
Consider first case (1). The inclusion $G_{\ell^\infty} \subseteq C_\ell$ is part of corollary \ref{cor_EC}, and implies $G_\infty \subseteq \prod_\ell G_{\ell^\infty} \subseteq \prod_\ell C_\ell$. In particular, $G_\infty$ is abelian, so class field theory allows us to interpret $\rho_\infty$ as a map
\[
I_K \xrightarrow{\rho_\infty} \prod_\ell C_\ell
\]
that is trivial on $K^*$. As in the proof of theorem \ref{thm_Finale}, since we are looking for a \textit{lower} bound on $G_\infty$ no harm is done in replacing $I_K$ by the group of idèles of the Hilbert class field of $K$; concretely, this means considering the restriction of $\rho_\infty$ to $\displaystyle \prod_{v \in \Omega_K} \mathcal{O}_{K,v}^\times$, where $\Omega_K$ is the set of finite places of $K$. Recall from theorem \ref{thm_FundamentalTheoremCM} that the action of $\rho_\infty$ on a finite idèle $a=(a_v)_{v \in \Omega_K}$ is given by
\[
\rho_\infty (a)=\varepsilon(a) \left( N_{K_\ell/E_\ell}(a_\ell^{-1}) \right)_{\ell \text{ prime}}.
\]
As in the proof of theorem \ref{thm_Finale}, if we let $\mu(E)$ be the group of roots of unity in $E$ we know that $\ker \varepsilon$ is a subgroup of $\prod_{v \in \Omega_K} \mathcal{O}_{K,v}^\times$ of index at most $|\mu(E)|$, and since $E$ is a quadratic imaginary field we have $|\mu(E)|\leq 6$. Therefore the image of $\rho_\infty$ has index at most $|\ker \varepsilon| \leq 6$ in the image of the map

\[
\begin{array}{cccc}
\varphi_\infty : & \prod_{v \in \Omega_K} \mathcal{O}_v^\times & \to & \prod_\ell (\mathcal{O}_E \otimes \mathbb{Z}_\ell)^\times = \prod_\ell C_\ell \\
 & (a)_v & \mapsto & \left( N_{K_\ell/E_\ell}\left(a_\ell \right) \right)_\ell
\end{array}
\]
given by taking local norms from $K_\ell$ to $E_\ell$. Hence in particular we have $\left[\prod_\ell C_\ell : G_\infty \right] \leq 6 \left[\prod_\ell C_\ell:\operatorname{Im} \varphi_\infty \right]$, and it suffices to show that
\[
\left[ \prod_\ell C_\ell : \operatorname{Im} \varphi_\infty \right] \leq [K:E] = \frac{1}{2}[K:\mathbb{Q}],
\]
which follows from \cite[Theorem 7 on p.~161]{MR2467155} (the global field counterpart of theorem \ref{thm_CokernelOfNorm}). The remaining assertion of part (1) is exactly the content of corollary \ref{cor_EC}.

\smallskip

As for part (2), we have seen in lemmas \ref{lemma_ContainmentEC} and \ref{lemma_NontrivialIntersection} that in this case $G_{\ell^\infty}$ is contained in $N(C_\ell)$, but not in $C_\ell$. If we let $K^1=K \cdot E$, then $A$ admits complex multiplication by $E$ over $K^1$, so $\rho_\infty\left( \abGal{K^1} \right)$ is contained in $\prod_\ell C_\ell$ by part (1). Since $\abGal{K^1}$ has index 2 in $\abGal{K}$ we must have $H_\infty = \rho_\infty\left( \abGal{K^1} \right)$, so that the index $[G_\infty:H_\infty]$ is indeed 2 and applying part (1) we find $\left[\prod_\ell C_\ell : H_\infty \right] \leq 3 [K^1:\mathbb{Q}]=6[K:\mathbb{Q}]$; moreover, the index $[\prod_\ell N(C_\ell):G_\infty]$ is not finite since the same is clearly true for the index $[\prod_\ell N(C_\ell):\prod_\ell C_\ell]$.
Finally, if $\ell$ is unramified in $K^1$ we see from corollary \ref{cor_EC} (applied to $A/K^1$) that $G_{\ell^\infty}$ contains all of $C_\ell$, and by lemma \ref{lemma_NontrivialIntersection} we know that $G_{\ell^\infty}$ also contains an element of $N(C_\ell) \setminus C_\ell$. The equality $G_{\ell^\infty}=N(C_\ell)$ then follows from lemma \ref{lemma_NormalizerCartan}.

\smallskip

As for the last assertion, notice that if we exclude elliptic curves with $j$-invariant equal to $0$ or $1728$ the field of complex multiplication $E$ is neither $\mathbb{Q}(i)$ nor $\mathbb{Q}(\zeta_3)$, so the only roots of unity in $E$ are $\pm 1$. This implies that $\operatorname{ker} \varepsilon$ has index at most 2 in $\prod_{v \in \Omega_K} \mathcal{O}_v^\times$, and the same argument as above shows that the constants $3$ and $6$ can indeed be replaced by $1$ and $2$.
\end{proof}

\begin{remark}
The following simple example shows that the constants 3 and 6 cannot be improved in general. We consider the elliptic curve $A$ over the field $K=\mathbb{Q}\left(\zeta_3 \right)$ given by the Weierstrass equation $y^2=x^3+1$. As it is clear, $A$ has complex multiplication (over $K$) by the full ring of integers of $E=K$. Moreover, all the 2-torsion points of $A$ are defined over $K$, so $G_2$ has trivial reduction modulo 2. Hence $G_2$ is a subgroup of $\ker\left( \mathbb{Z}_2[\zeta_3]^\times \to \mathbb{F}_2[\zeta_3]^\times \right)$, and its index in $(\mathcal{O}_E \otimes \mathbb{Z}_2)^\times \cong \mathbb{Z}_2[\zeta_3]^\times$ is divisible by $\left|\mathbb{F}_2[\zeta_3]^\times \right|=3$. Likewise, the fact that the 3-torsion point with coordinates $(0,1)$ is defined over $K$ shows that the index of $G_3$ in $(\mathcal{O}_E \otimes \mathbb{Z}_3)^\times$ is divisible by 2. Thus we conclude that the index of $G_\infty$ in $\prod_\ell C_\ell$ is at least $6=3[K:\mathbb{Q}]$, so that the constant 3 is indeed sharp. Finally, considering the $\mathbb{Q}$-elliptic curve given by the same Weierstrass equation shows the optimality of part (2): in this case $H_\infty$ is exactly the image of the Galois representation attached to $A/K$, so we have $[\prod_\ell C_\ell : H_\infty]=6$ by what we just showed.
\end{remark}

\subsection{Abelian surfaces}\label{subsect_Surfaces}
An easy direct computation shows that when $\dim A = 2$ the kernel of the reflex norm is always connected, and therefore the group $F$ of theorem \ref{thm_Finale} is trivial. Since furthermore simple CM types are automatically non-degenerate in dimension 2, combining theorems \ref{thm_Finale} and \ref{thm_Nondegenerate} we deduce:
\begin{corollary}
Let $A/K$ be an absolutely simple abelian variety of dimension 2. Suppose that $A$ has CM over $K$ by the field $E$ and let $\ell$ be a prime number such that $A$ has good reduction at all places of $K$ of characteristic $\ell$. The group $G_{\ell^\infty} \cap \MT(A)(\mathbb{Z}_\ell)$ has index at most $[\fieldExtension]$ in $\MT(A)(\mathbb{Z}_\ell)$, hence we have $\displaystyle [K(A[\ell^n]):K] \geq \frac{1}{[\fieldExtension]} (1-1/\ell)^3 \ell^{3n}$ for $\ell \neq 2$, while for $\ell=2$ we have $\displaystyle [K(A[2^n]):K] \geq \frac{1}{2^7[\fieldExtension]} 2^{3n}$. Finally, if $\ell$ is unramified in $K \cdot E$ we even have $[K(A[\ell^n]):K] \geq (1-1/\ell)^3 \ell^{3n}$.
\end{corollary}

\section{A family of varieties with small 2-torsion fields}\label{sect_Examples}
Let $p \geq 3$ be a prime number and $K_p$ be the cyclotomic field $\mathbb{Q}\left( \zeta_{p} \right)$. We let $C_p$ be the unique smooth $K_p$-curve birational to $y^{p}=x(1-x)$ and $J(p)$ be its Jacobian, again over $K_p$. It is clear that $C_p$ admits an action of $\mu_{p}$, so $J(p)$ is a CM abelian variety, admitting complex multiplication over $K_p$ by the full ring of integers of $K_p$. Notice furthermore that $C_p$ is birational to the curve
\[
z^2=w^{p}+1/4
\]
(just set $x=z+1/2$, $y=-w$), so it is hyperelliptic of genus $\frac{p-1}{2}$. Direct inspection of the model $y^{p}=x(1-x)$ reveals that $C_p$ is smooth away from $p$, so $J(p)$ has everywhere good reduction over $K_p$ except perhaps at the unique place dividing $p$. The reflex field is $K_p^*=K_p$. Let us compute the CM type $S$ of $J(p)$: in the basis $\omega_j:=w^j \frac{dw}{z}$ ($j=0,\ldots,\frac{p-3}{2}$) of the space of differentials on $C_p$, the action of $\zeta_p$ is given by
$
[\zeta_p]^* \omega_j = \zeta_p^{j+1} \omega_j,
$
hence the CM type, considered as a subset of $\left(\frac{\mathbb{Z}}{p\mathbb{Z}} \right)^\times$, is $\left\{1,\ldots,\frac{p-1}{2}\right\}$. Equivalently, 
\[
S=\left\{ g \in \left(\frac{\mathbb{Z}}{p\mathbb{Z}} \right)^\times \bigm\vert 2\langle g \rangle < p \right\},
\]
where $\langle g \rangle$ is the unique integer lying in the interval $[0,p-1]$ that is congruent to $g$ modulo $p$. This description shows that our CM type agrees with the type $S_1$ described in \cite{Mai1989192}, which by \cite[Lemma 1]{Mai1989192} is nondegenerate (cf. also \cite{MR0190144}): thus we have $\operatorname{rank} \MT(A)=\dim A +1=\frac{p+1}{2}$. 

Let now $\beta_1, \ldots, \beta_{p}$ be the roots of $w^{p}+1/4=0$ in $\overline{\mathbb{Q}}$, and let $P_i=(\beta_i,0)$ be the corresponding points of $C_p$ (in the coordinates $(w,z)$). Finally, for $i=1,\ldots,p$ let $d_i$ denote the divisor $(P_i)-(\infty)$. It is known (see for example \cite[§5.1]{MR3156850}) that the 2-torsion of $J(p)$ is an $\mathbb{F}_2$-vector space of dimension $p-1$ spanned by the $d_i$'s, which are only subject to the linear relation $\sum_{i=1}^p d_i=[0]$. It follows that the 2-torsion field $K_p(J(p)[2])=K_p\left( \left\{\beta_i\right\} \right)=K_p\left( \sqrt[p]{1/4} \right)$ has degree $p$ over $K_p$, so for $\ell=2$ and $n=1$ the ratio $\ell^{n \operatorname{rank} \MT(A)} \bigm/ [K(A[\ell^n]):K]$ is given by
\[
\frac{2^{\operatorname{rank} \MT(A)}}{[K(J(p)[2]):K]} = \frac{2^{(p+1)/2}}{p} = \frac{2^{\dim J(p)+1}}{2 \dim J(p)+1},
\]
which shows in particular that, as claimed in the introduction, the optimal bound on the quantity $\ell^{n \operatorname{rank} \MT(A)} \bigm/ [K(A[\ell^n]):K]$ grows at least exponentially in the dimension of $A$.

\bibliographystyle{plain}
\bibliography{Biblio}

\end{document}